\documentclass[reqno]{amsart}
\usepackage[left=2.6cm,right=2.6cm,top=3cm,bottom=3cm]{geometry}
\usepackage[backref=page]{hyperref}
\usepackage[nameinlink,capitalize]{cleveref}
\usepackage{cite}
\usepackage{enumitem}
\usepackage{graphicx}
\usepackage{caption}
\usepackage{subcaption}

\theoremstyle{plain}
\newtheorem{theorem}{Theorem}

\newtheorem{lemma}[theorem]{Lemma}

\theoremstyle{definition}

\theoremstyle{remark}
\newtheorem{remark}{Remark}

\allowdisplaybreaks

\begin{document}

\author{Phuoc Vinh Dinh}
\address{Phuoc Vinh Dinh$^{1,2,3}$ (ORCID: 0009-0000-6297-7798)\newline
	$^1$Faculty of Applied Sciences, Ho Chi Minh City University of Technology (HCMUT), 268 Ly Thuong Kiet Street, Dien Hong Ward, Ho Chi Minh City, Vietnam;\newline
	$^2$Vietnam National University, Ho Chi Minh City, Vietnam\newline
	$^3$Department of Mathematics, FPT University, Ho Chi Minh City, Vietnam}
\email{dpvinh.sdh251@hcmut.edu.vn, vinhdp2@fe.edu.vn}

\author{Phuong Le}
\address{Phuong Le$^{4,2}$ (ORCID: 0000-0003-4724-7118)\newline
	$^4$Faculty of Economic Mathematics, University of Economics and Law, Ho Chi Minh City, Vietnam; \newline
	$^2$Vietnam National University, Ho Chi Minh City, Vietnam}
\email{phuongl@uel.edu.vn}

\author{Tien Dung Nguyen}
\address{Tien Dung Nguyen$^{1,2}$\newline
	$^1$Faculty of Applied Sciences, Ho Chi Minh City University of Technology (HCMUT), 268 Ly Thuong Kiet Street, Dien Hong Ward, Ho Chi Minh City, Vietnam;\newline
	$^2$Vietnam National University, Ho Chi Minh City, Vietnam}
\email{dungnt@hcmut.edu.vn}

\title[Spreading dynamics under shifting climate]{Spreading--vanishing dynamics and speed selection in a free-boundary competition model under a shifting climate}

\date{\today}

\begin{abstract}
	We study a free boundary problem for a diffusive Lotka--Volterra competition system describing the invasion of a new species into the habitat of a native competitor, where habitat suitability shifts from unfavourable to favourable at a constant speed $c>0$ due to climate change. Only the invader is affected by the shifting environment and only its range is governed by a Stefan-type free boundary, while the native species occupies the whole half line. We work throughout in the weak competition regime, in which the two species may coexist. We prove a spreading--vanishing dichotomy: either the invader spreads and the pair converges to the coexistence steady state $(u^*,v^*)$, or the invader vanishes and the native species recovers its carrying capacity. In the vanishing case we obtain the explicit bound $\lim_{t\to\infty}h(t)\le\frac{\pi}{2}\sqrt{d_1c_2/(a_1c_2-a_2c_1)}$, and we give criteria guaranteeing each alternative. When spreading occurs, we determine the exact asymptotic spreading speed: $\lim_{t\to\infty}h(t)/t=\min\{c,c_0\}$, where $c_0$ is the spreading speed of the corresponding homogeneous weak competition system. In particular the invasion is slowed down both by the competitor and by the climate shift, and the slower of the two mechanisms is the one that determines the speed. Numerical simulations illustrate the results.
\end{abstract}


\dedicatory{}


\maketitle

\section{Introduction and statement of the results}\label{sec_1}

\subsection{Background}

Understanding how species respond to a changing climate is a central question in
spatial ecology, and reaction--diffusion equations have long been used to model
the shift of habitats and the propagation of populations; see
\cite{MR2471053,MR3691993,MR3377515} and the references therein. When a species
expands into new territory, the boundary of its range is not known in advance,
which naturally leads to free boundary problems. The starting point of this line
of research is the work of Du and Lin \cite{MR2607347}, where a diffusive
logistic model with a Stefan-type free boundary was introduced to describe the
spreading of a single species. They established a sharp spreading--vanishing
dichotomy: either the species spreads over the whole space and stabilises at its
carrying capacity, or it vanishes; which alternative occurs is governed by the
initial habitat size and by the expansion capacity of the species. We refer to
the survey \cite{MR4404216} for an account of the many developments that
followed.

To take the non-stationarity of the environment into account, several authors
have coupled the free boundary framework with a shifting habitat
\cite{MR3764580,MR3871607,MR3639148,MR4766631}. Of particular relevance to the
present paper is the model of Hu, Hao, Song and Du \cite{HU20205931},
\begin{equation}\label{single}
	\begin{cases}
		u_t = d_1 u_{xx} + (A(x - ct) - b_1 u) u, & t > 0,~ 0 < x < h(t), \\
		u_x(t, 0) = 0, ~ u(t, x) = 0, & t > 0, ~ h(t) \leq x < +\infty, \\
		h'(t) = -\mu (A(h(t) - ct))u_x(t, h(t)), & t > 0, \\
		h(0)=h_0,~ u(0, x) = u_0(x), & 0 \leq x \leq h_0,
	\end{cases}
\end{equation}
in which the environment turns from \emph{unfavourable to favourable} at speed
$c>0$, so that the species benefits from the climate shift, a situation which
is particularly relevant in invasion ecology, since many invasive species take
advantage of climate warming. For \eqref{single} a spreading--vanishing dichotomy
holds and, when spreading occurs, there is a critical speed
$\tilde c_0\in(0,2\sqrt{a_1d_1})$ such that the spreading profile is a semi-wave
with forced speed $c$ when $c<\tilde c_0$, and the usual semi-wave with speed
$\tilde c_0$ when $c\ge \tilde c_0$; in both cases the range expands at the
asymptotic speed $\min\{c,\tilde c_0\}$. (The notation $\tilde c_0$ is used here
only to avoid a clash with the speed $c_0$ of the competition system, introduced
in \eqref{c0} below.)

A second, independent extension of \cite{MR2607347} concerns the interaction
between an invader and a native competitor. In a homogeneous environment, Du and
Lin \cite{MR3327894} studied the two-species free boundary problem in which the
invader $u$ occupies the varying interval $[0,h(t)]$ while the resident $v$ lives
on the whole half line $[0,\infty)$. They showed that an inferior invader always
dies out, whereas a superior invader obeys a spreading--vanishing dichotomy; the
exact spreading speed of a superior invader was subsequently determined by Du,
Wang and Zhou \cite{MR3609207} through a semi-wave analysis. The remaining
regime, \emph{weak competition}, in which the two species can coexist, was
treated by Wang, Nie and Du \cite{MR3986328}: a spreading--vanishing dichotomy
again holds, when the invasion succeeds both populations converge to the
coexistence state $(u^*,v^*)$, and the exact spreading speed exists and is
strictly smaller than the spreading speed of $u$ in the absence of $v$. Related
free boundary competition systems, including nonlocal ones and systems in which
both species have a free boundary, are studied in
\cite{Lei2018,MR4873231}.

The simultaneous effect of a shifting climate and of interspecific competition in
a free boundary setting was addressed by Lei, Nie, Dong and Du
\cite{Lei2018}, where both competitors expand through free boundaries in a common
shifting habitat. In many real invasions, however, the native species is already
established across the whole habitat and does not experience a moving range
boundary; only the invader faces a climate-driven frontier. Motivated by this
observation, in the companion paper \cite{DLN2026a} we studied
\begin{equation}\label{main}
	\begin{cases}
		u_t = d_1u_{xx} + (A(x - ct) - b_1u - c_1 v) u, & t > 0,~ 0 < x < h(t), \\
		v_t = d_2v_{xx} + (a_2 - b_2u - c_2v)v, & t > 0,~ 0 < x < +\infty, \\
		u_x(t, 0) = v_x(t, 0) = 0, ~ u(t, x) = 0, & t > 0, ~ h(t) \leq x < +\infty, \\
		h'(t) = -\mu (A(h(t) - ct))u_x(t, h(t)), & t > 0, \\
		h(0)=h_0,~ u(0, x) = u_0(x), & 0 \leq x \leq h_0, \\
		v(0, x) = v_0(x), & 0 \le x < +\infty,
	\end{cases}
\end{equation}
in the two \emph{weak--strong} competition regimes, that is, when the invader is
either an inferior or a superior competitor. The present paper is devoted to the
remaining case, and from the ecological point of view the most interesting,
case of \emph{weak competition}, in which neither species excludes the other and
coexistence is possible.

\subsection{The model and the standing assumptions}

In \eqref{main}, $u(t,x)$ and $v(t,x)$ denote the population densities of the
invasive and of the native species, and $x=h(t)$ is the free boundary, to be
determined together with $u$ and $v$. The invader exists initially on
$\{0\le x\le h_0\}$ and expands through the front $x=h(t)$, which obeys the
Stefan condition $h'(t)=-\mu(A(h(t)-ct))u_x(t,h(t))$; for biological
interpretations of Stefan-type free-boundary laws in climate-driven range models,
see, for example, \cite{MR3764580}. The native species diffuses and grows on the whole habitat
$[0,\infty)$. Both species satisfy a no-flux condition at $x=0$. The constants
$d_1,d_2$ are the diffusion rates, $a_2$ is the intrinsic growth rate of $v$,
$b_1,c_2$ are the intra-specific and $c_1,b_2$ the inter-specific competition
rates, all assumed positive; $h_0>0$.

The intrinsic growth rate of the invader is $A(x-ct)$: the environment is shifted
at the constant speed $c>0$ in the direction of increasing $x$. Here $A$ should be
read as describing the \emph{climatic range limit of the invader}, which is what
moves, rather than climate change acting on the community as a whole: the native
competitor is already established throughout $[0,\infty)$ and, on the scales
relevant to the invasion, sits well inside its own climatic envelope, so that its
growth rate is effectively constant. See \cite{DLN2026a} for a fuller
discussion, and \cite{Lei2018} for the complementary situation in which both
species are subject to a shifting habitat. Throughout the
paper we assume, as in \cite{HU20205931,DLN2026a}, that $A$ is Lipschitz
continuous on $\mathbb{R}$, strictly decreasing on $[0,l_0]$, and
\begin{equation}\label{A}
	A(\xi) =
	\begin{cases}
		a_1 & \text{ if } \xi \leq 0,\\
		a_0 & \text{ if } \xi \geq l_0,
	\end{cases}
\end{equation}
where $l_0 > 0$ and $a_0 < 0 < a_1$. Thus the region $\{x<ct\}$ is favourable for
the invader and the region $\{x>ct+l_0\}$ is unfavourable, and the favourable
region advances at speed $c$. The function $\mu$, which measures how easily the
range boundary expands, is Lipschitz continuous and nondecreasing on $[a_0,a_1]$
with
\begin{equation}\label{mu}
	0 < \mu(a_0) \le \mu(a_1).
\end{equation}
It is convenient (and harmless) to extend $\mu$ to all of $\mathbb{R}$ by setting
$\mu(s):=\mu(a_0)$ for $s<a_0$ and $\mu(s):=\mu(a_1)$ for $s>a_1$; the extended
function is again Lipschitz and nondecreasing, and this extension is used only in
the construction of the perturbed semi-waves in \cref{sec_5}.

The initial functions satisfy
\begin{equation}\label{initial}
	\begin{cases}
		u_0 \in C^2([0,h_0]),& u_0 > 0 \text{ in } [0,h_0), ~ u_0'(0) = u_0(h_0) = 0,\\
		v_0 \in C_b^2([0,+\infty)),& v_0 > 0 \text{ in } [0,+\infty), ~ v_0'(0) = 0,\\
		&\displaystyle\liminf_{x\to+\infty} v_0(x) > 0.
	\end{cases}
\end{equation}
Here $C_b^2([0,+\infty))$ denotes the space of $C^2$ functions whose
function and first two derivatives are bounded. The last requirement in
\eqref{initial} says that the native species is already established over the
whole habitat; together with continuity and positivity it yields
\begin{equation}\label{v0inf}
	m_0:=\inf_{x\ge0} v_0(x) > 0 .
\end{equation}
A positive-tail assumption of this type is standard in homogeneous
free-boundary competition models; compare \cite{MR3986328,MR3609207}. In the
present paper it is used only in \cref{sec_6}, where a uniform lower bound for
the native population far ahead of the moving front is needed.

When the native competitor is absent ($v\equiv0$), \eqref{main} reduces to
\eqref{single}, studied in \cite{HU20205931}. When the shifting term is replaced
by the constant $a_1$, \eqref{main} reduces to the homogeneous competition system
of Du and Lin \cite{MR3327894},
\begin{equation}\label{main2}
	\begin{cases}
		u_t = d_1u_{xx} + (a_1 - b_1u - c_1 v) u, & t > 0,~ 0 < x < h(t), \\
		v_t = d_2v_{xx} + (a_2 - b_2u - c_2v)v, & t > 0,~ 0 < x < +\infty, \\
		u_x(t, 0) = v_x(t, 0) = 0, ~ u(t, x) = 0, & t > 0, ~ h(t) \leq x < +\infty, \\
		h'(t) = -\mu (a_1)u_x(t, h(t)), & t > 0, \\
		h(0)=h_0,~ u(0, x) = u_0(x), & 0 \leq x \leq h_0, \\
		v(0, x) = v_0(x), & 0 \le x < +\infty.
	\end{cases}
\end{equation}
Problem \eqref{main2} was analysed in \cite{MR3327894,MR3609207} under the
weak--strong competition conditions
$\frac{a_1}{a_2} < \min\{\frac{b_1}{b_2},\frac{c_1}{c_2}\}$ (inferior invader) or
$\frac{a_1}{a_2} > \max\{\frac{b_1}{b_2},\frac{c_1}{c_2}\}$ (superior invader),
and in \cite{MR3986328} under the weak competition condition
\begin{equation}\label{weak_competition}
	\frac{c_1}{c_2} < \frac{a_1}{a_2} < \frac{b_1}{b_2},
\end{equation}
which we assume from now on. Under \eqref{weak_competition} the kinetic system
associated with \eqref{main2} has the globally attracting coexistence state
\begin{equation}\label{coexist}
	u^* := \frac{a_1c_2-a_2c_1}{b_1c_2-b_2c_1},\qquad
	v^* := \frac{a_2b_1-a_1b_2}{b_1c_2-b_2c_1},
\end{equation}
and \eqref{weak_competition} guarantees $b_1c_2-b_2c_1>0$, $u^*>0$, $v^*>0$ and
\begin{equation}\label{ineq_star}
	0<u^*<\frac{a_1}{b_1},\qquad 0<v^*<\frac{a_2}{c_2},\qquad
	a_1-\frac{a_2c_1}{c_2}>0 .
\end{equation}
Throughout the paper we write
\begin{equation}\label{Rstar}
	R^*:=\frac{\pi}{2}\sqrt{\frac{d_1c_2}{a_1c_2-a_2c_1}},
	\qquad
	R_0:=\frac{\pi}{2}\sqrt{\frac{d_1}{a_1}},
\end{equation}
so that $0<R_0<R^*$ by \eqref{ineq_star}.

The weak competition regime is not a routine interpolation between the two
weak--strong cases. Behind a successful invasion front the limiting state is
the coexistence equilibrium $(u^*,v^*)$, so the front is governed by a genuinely
two-component semi-wave connecting coexistence to the resident-only state.
Moreover, in the shifting environment the corresponding inhomogeneous semi-wave
need not be unique. Our spreading-speed proof therefore does not use a
selection or uniqueness theorem for that family: instead we establish the two
endpoint limits of its boundary slope and use them to construct upper and lower
barriers satisfying the Stefan condition as inequalities. This is the key
mechanism behind the exact formula in \cref{th:spreading_speed}.

\subsection{Main results}

Our first result describes the long-time dynamics of \eqref{main}. Recall from
\cref{th6} below that $h$ is strictly increasing, so that
$h_\infty:=\lim_{t\to\infty}h(t)\in(h_0,+\infty]$ always exists.

\begin{theorem}\label{th:dynamics}
	Let $(u, v, h)$ be the unique solution of \eqref{main} with initial functions
	satisfying \eqref{initial}, and suppose that \eqref{weak_competition} holds.
	Then exactly one of the following alternatives occurs.
	\begin{enumerate}
		\item[(i)] \textit{Spreading of $u$}:
		\[
		h_\infty = +\infty \quad\text{ and }\quad \lim_{t\to\infty} (u(t,\cdot), v(t,\cdot)) = (u^*, v^*) \text{ in } C^2_{\rm loc}([0, \infty));
		\]
		\item[(ii)] \textit{Vanishing of $u$}:
		\[
		h_\infty \le R^* \quad\text{ and }\quad
		\lim_{t\to\infty} \|u(t,\cdot)\|_{C([0,h(t)])}=0,\quad
		\lim_{t\to\infty} v(t,\cdot) = \frac{a_2}{c_2} \text{ in } C^2_{\rm loc}([0, \infty)).
		\]
	\end{enumerate}
\end{theorem}

The next result gives criteria under which each of the two alternatives takes
place. It is the analogue, for the competition system \eqref{main}, of
\cite[Theorem 1.5]{HU20205931} and of \cite[Theorem 1.2]{MR3986328}.

\begin{theorem}\label{th:criteria}
	Assume \eqref{weak_competition} and let $(u,v,h)$ be the solution of
	\eqref{main} with initial functions satisfying \eqref{initial}.
	\begin{enumerate}
		\item[(i)] If $h_0\ge R^*$, then spreading always happens.
		\item[(ii)] Fix $h_0<R_0$ and $v_0$, write $u_0=\sigma\phi$ where $\phi$
		satisfies the conditions imposed on $u_0$ in \eqref{initial} and
		$\sigma>0$, and denote by $(u^\sigma,v^\sigma,h^\sigma)$ the corresponding
		solution. Then there exists
		$\sigma_0=\sigma_0(h_0,\phi,v_0,c)\in(0,+\infty]$ such that vanishing
		happens for $0<\sigma\le\sigma_0$ and spreading happens for
		$\sigma>\sigma_0$.
	\end{enumerate}
\end{theorem}

\begin{remark}\label{rm:criteria}
	As in \cite[Remark 1.6]{HU20205931} we do not know whether
	$\sigma_0=+\infty$ can actually occur; this remains an open problem already
	for the single species model \eqref{single}.
\end{remark}

To state our second main result we recall from \cite[Theorem 1.5]{MR3986328}
(see \cref{th:semi-wave0} below) that, under \eqref{weak_competition}, the
homogeneous problem \eqref{main2} possesses a well-defined asymptotic spreading
speed $c_0>0$, characterised by the semi-wave identity
$-\mu(a_1)\Psi_{c_0}'(0)=c_0$.

\begin{theorem}\label{th:spreading_speed}
	Assume \eqref{weak_competition} and \eqref{initial}. If alternative (i) of
	\cref{th:dynamics} occurs, then
	\[
	\lim_{t\to\infty} \frac{h(t)}{t} = \min\{c, c_0\}.
	\]
\end{theorem}

Let $c_0^{\rm sing}$ denote the asymptotic spreading speed of the invader in a
globally favourable environment \emph{in the absence of the native species},
that is, the speed at which the free boundary of
\begin{equation}\label{single_hom}
	\begin{cases}
		\hat u_t = d_1\hat u_{xx}+(a_1-b_1\hat u)\hat u, & t>0,\ 0<x<\hat h(t),\\
		\hat u_x(t,0)=\hat u(t,\hat h(t))=0,\quad \hat h'(t)=-\mu(a_1)\hat u_x(t,\hat h(t)), & t>0,\\
		\hat h(0)=h_0,\quad \hat u(0,\cdot)=u_0
	\end{cases}
\end{equation}
advances when spreading occurs; see \cite{MR2607347,MR4404216}.

\begin{lemma}\label{lm:c0sing}
	Assume \eqref{weak_competition}. Then $c_0\le c_0^{\rm sing}$.
\end{lemma}

\begin{proof}
	Choose admissible initial data with $h_0\ge R^*$, so that spreading occurs
	for \eqref{main2} by \cite[Theorem 1.2]{MR3986328}. Let
	$(\tilde u,\tilde v,\tilde h)$ be the corresponding solution of
	\eqref{main2}, and let $(\hat u,\hat h)$ solve \eqref{single_hom} with the
	same $u_0$ and $h_0$. Since $\tilde v>0$ we have
	$\tilde u_t\le d_1\tilde u_{xx}+(a_1-b_1\tilde u)\tilde u$, while
	$\tilde u_x(t,0)=0$ and $\tilde h'(t)=-\mu(a_1)\tilde u_x(t,\tilde h(t))$;
	thus $(\tilde u,\tilde h)$ is a lower solution of \eqref{single_hom} and the
	scalar comparison principle gives $\tilde h\le\hat h$ for all $t>0$.
	Since $\tilde h_\infty=+\infty$, also $\hat h_\infty=+\infty$.
	The spreading-speed limits for the two homogeneous problems therefore give,
	after division by $t$ and passage to the limit,
	$c_0\le c_0^{\rm sing}$.
\end{proof}

\begin{remark}\label{rm:speed}
	The formula $\min\{c,c_0\}$ has a transparent ecological meaning. The number
	$c_0$ is the speed at which the invader would advance in a globally
	favourable environment while competing with the native species, whereas $c$
	is the speed at which the favourable habitat itself moves. If $c_0\le c$ the
	invasion front travels at its own intrinsic speed and does not feel the
	climate constraint; if $c_0>c$ the front cannot outrun the moving edge of the
	favourable region and is slowed down to the speed $c$ of the climate shift.
	Note moreover that, by \cref{lm:c0sing}, $c_0\le c_0^{\rm sing}$, and that the
	inequality is in general strict: the native species depresses the invader
	precisely at its leading edge, where $v$ is close to its carrying capacity
	$a_2/c_2$. For the parameters used in \cref{sec_7}, a direct numerical
	solution of the homogeneous semi-wave boundary-value problems gives
	$c_0\approx0.01208$ and $c_0^{\rm sing}\approx0.03872$, so the latter is
	about $3.2$ times larger. Hence in the present model the invasion is retarded by two independent
	mechanisms, interspecific competition and the pace of climate change, and
	it is the more restrictive of the two that determines the observed speed.
\end{remark}

\begin{remark}\label{rm:compare}
	It is instructive to compare \cref{th:dynamics} with the corresponding result
	of \cite{DLN2026a} for a superior invader. There, when spreading occurs, the
	native species is driven to extinction and $(u,v)\to(a_1/b_1,0)$, whereas here
	the two species coexist and $(u,v)\to(u^*,v^*)$. This difference is not
	merely cosmetic: it changes the structure of the semi-waves that govern the
	front. In the superior case the $v$-component of the semi-wave vanishes at
	$-\infty$, and the construction of a lower solution requires a delicate
	modification of the profile far behind the front (see
	\cite[Lemma 3.3]{MR3609207} and the corresponding lemma in \cite{DLN2026a}). In
	the weak competition case the $v$-component tends to $v^*>0$ at $-\infty$,
	so no such modification is needed; on the other hand the state at $-\infty$ is
	now a genuine two-species coexistence equilibrium, and the asymptotic analysis
	of the semi-wave near $-\infty$ involves the full four-dimensional linearised
	system rather than a scalar one (see \cref{lm:asym}). A further difference
	appears in the vanishing case: since $u$ no longer eliminates $v$, the limit
	of $v$ is its own carrying capacity $a_2/c_2$ and not $0$.
\end{remark}

\begin{remark}\label{rm:hinfty}
	The bound $h_\infty\le R^*$ in \cref{th:dynamics}(ii) is sharper than the mere
	finiteness of $h_\infty$ and is what makes \cref{th:criteria} and the
	borderline case $c=c_0$ of \cref{th:spreading_speed} accessible; see
	\cref{lm:eq}. It is obtained from the following simple but useful observation
	(\cref{lem:reduction}): if $h_\infty<+\infty$, then $h(t)<ct$ for all large
	$t$, so that after a finite time the shifting environment is no longer felt
	and \eqref{main} reduces \emph{exactly} to the homogeneous system
	\eqref{main2}. The same observation is the key to the case $c>c_0$.
\end{remark}

\textbf{Organisation of the paper.} \Cref{sec_2} collects the basic results on
\eqref{main} (existence, uniqueness, global bounds) and proves, in full, the
comparison principle in the form used later; the details of the local existence
theory, which follow \cite[Theorem 2.1]{MR3327894}, are deferred to
\cref{sec_app}. \Cref{sec_3} contains the proof
of \cref{th:dynamics}, and \cref{sec_4} that of \cref{th:criteria}. \Cref{sec_5}
is devoted to the semi-waves associated with the shifting environment; this is
the technical core of the paper. \Cref{sec_6} proves \cref{th:spreading_speed}.
\Cref{sec_7} presents numerical simulations, and \cref{sec_8} concludes with some
comments and open questions.

\section{Preliminary results}\label{sec_2}

\subsection{Existence, uniqueness and global bounds}

\begin{theorem}\label{th4}
	For any given $\alpha \in (0, 1)$ and any $(u_0, v_0)$ satisfying
	\eqref{initial}, there exists a $T > 0$ such that problem \eqref{main} admits
	a unique bounded solution
	$$
	(u, v, h) \in C^{(1+\alpha)/2, 1 + \alpha}(D_T) \times C^{(1+\alpha)/2, 1 + \alpha}(D_T^{\infty}) \times C^{1 + \alpha/2}([0,T]).
	$$
	Moreover,
	$$
	\|u\|_{C^{(1+\alpha)/2, 1 + \alpha}(D_T)} + \|v\|_{C^{(1+\alpha)/2, 1 + \alpha}(D_T^{\infty})} + \|h\|_{C^{1 + \alpha/2}([0,T])} \leq C,
	$$
	where $D_T = \{(t,x) \in \mathbb{R}^2: t \in [0,T], x \in [0, h(t)]\}$,
	$D_T^{\infty} = \{(t,x)\in \mathbb{R}^2: t \in [0,T], x \in [0, +\infty)\}$, and $C$
	and $T$ depend only on $\alpha, h_0, c, \|u_0\|_{C^2([0,h_0])},
	\|v_0\|_{C_b^2([0,\infty))}$, on $a_0,a_1,a_2,b_1,b_2,c_1,c_2,d_1,d_2$ and on
	the Lipschitz constants of $A$ and $\mu$. Here $C^{(1+\alpha)/2,1+\alpha}$
	denotes the parabolic H\"older space of functions that are H\"older continuous
	with exponent $\frac{1+\alpha}{2}$ in $t$ and $1+\alpha$ in $x$, in the sense
	of \cite[Chapter IV]{MR241822}.
\end{theorem}

The proof is obtained by straightening the free boundary and applying the
contraction mapping theorem, exactly as in \cite[Theorem 2.1]{MR3327894}; the
only modification is the presence of the Lipschitz functions $A(\cdot-ct)$ and
$\mu(A(\cdot))$, which affects the constants but not the structure of the
argument. For the reader's convenience the details are given in \cref{sec_app}.

\begin{theorem}\label{th6}
	Problem \eqref{main} admits a unique solution $(u,v,h)$, defined for all
	$t>0$, and there exist positive constants $M_1,M_2,M_3$ such that
	\[
	0 < u(t,x) \leq M_1 \ \text{ for } (t, x) \in (0, +\infty) \times [0, h(t)),\qquad
	0 < v(t,x) \leq M_2 \ \text{ for } (t, x) \in (0, +\infty) \times [0, +\infty),
	\]
	\[
	0 < h'(t) \leq M_3 \ \text{ for } t \in (0,+\infty).
	\]
	Moreover $M_1=\max\{\|u_0\|_\infty,a_1/b_1\}$ and
	$M_2=\max\{\|v_0\|_\infty,a_2/c_2\}$ may be taken.
\end{theorem}

\begin{proof}
	Let $[0,T_{\max})$ be the maximal interval of existence given by \cref{th4}.
	Since $A(\cdot)\le a_1$, the function $u$ satisfies
	$u_t-d_1u_{xx}\le u(a_1-b_1u)$ in $\{0<x<h(t)\}$ together with
	$u_x(t,0)=0$ and $u(t,h(t))=0$; comparing $u$ with the solution
	$\overline{u}(t)$ of the logistic equation
	\[
	\frac{{\rm d}\overline{u}}{{\rm d}t} = \overline{u}(a_1 - b_1\overline{u}),
	\qquad \overline{u}(0)=\|u_0\|_\infty,
	\qquad\text{i.e.}\qquad
	\overline{u}(t)=\frac{a_1e^{a_1t}}{b_1\left(e^{a_1t}-1+\frac{a_1}{b_1\|u_0\|_\infty}\right)},
	\]
	we obtain $u\le \sup_{t\ge0}\overline{u}(t)=M_1$. Similarly, $v$ satisfies
	$v_t-d_2v_{xx}\le v(a_2-c_2v)$ on $[0,\infty)$, whence $v\le M_2$. The strong
	maximum principle and the Hopf boundary lemma give $u>0$ in
	$\{0\le x<h(t)\}$, $u_x(t,h(t))<0$ and hence $h'(t)>0$, and also $v>0$.

	To bound $h'$ from above, fix $M\ge\max\{h_0^{-1},\sqrt{(a_1+c_1M_2)/(2d_1)},\,
	4\|u_0\|_{C^1([0,h_0])}/(3M_1)\}$ and set
	$\Omega:=\{(t,x):0<t<T_{\max},\,h(t)-M^{-1}<x<h(t)\}$ and
	$\overline{U}(t,x):=M_1[2M(h(t)-x)-M^2(h(t)-x)^2]$. A direct computation using
	$h'>0$ shows that $\overline{U}_t-d_1\overline{U}_{xx}\ge2d_1M_1M^2\ge
	(A(x-ct)-b_1u-c_1v)u$ in $\Omega$, while
	$\overline{U}(t,h(t)-M^{-1})=M_1\ge u$, $\overline{U}(t,h(t))=0=u(t,h(t))$ and
	$\overline{U}(0,\cdot)\ge u_0$ on $[h_0-M^{-1},h_0]$. The maximum principle
	yields $u\le\overline{U}$ in $\Omega$, hence
	$u_x(t,h(t))\ge\overline{U}_x(t,h(t))=-2MM_1$ and
	\[
	0<h'(t)=-\mu(A(h(t)-ct))u_x(t,h(t))\le 2MM_1\mu(a_1)=:M_3 .
	\]
	Since $M_1,M_2,M_3$ do not depend on $T_{\max}$, standard parabolic
	regularity and \cref{th4} allow the solution to be continued beyond
	$T_{\max}$ unless $T_{\max}=\infty$; therefore $T_{\max}=\infty$.
\end{proof}

\subsection{Comparison principle}

The following comparison principle is the two-species analogue of
\cite[Lemma 2.6]{MR3327894}; note the \emph{interleaved} structure, typical of
competitive systems: an upper solution for $u$ is paired with a lower solution
for $v$, and vice versa. The two alternative conditions at $x=0$ are both needed:
the first is used for the upper solutions of \cref{sec_6}, the second for the
lower solutions, whose $u$-component is decreasing in $x$ and therefore does not
satisfy a one-sided Neumann condition at the origin.

\begin{theorem}\label{th7}
	Let $T \in (0,+\infty)$, and let
	\begin{align*}
		&\underline{h}, \overline{h} \in C^1([0,T]),\quad \underline h(0)>0,\\
		&\underline{u} \in C(\overline{U_T}) \cap C^{1,2}(U_T) \text{ with } U_T := \{ (t,x) \in \mathbb{R}^2 : t \in (0, T], x \in (0, \underline{h}(t))\},\\
		&\overline{u} \in C(\overline{V_T}) \cap C^{1,2}(V_T) \text{ with } V_T := \{ (t,x) \in \mathbb{R}^2 : t \in (0, T], x \in (0, \overline{h}(t))\},\\
		&\underline{v}, \overline{v} \in (L^{\infty} \cap C)([0,T] \times [0,\infty)) \cap C^{1,2}((0,T] \times [0,\infty)),\\
		&\underline{v}, \overline{v}, \underline{u}, \overline{u} \ge 0 .
	\end{align*}
	Let $(u,v,h)$ be the unique solution of \eqref{main}.

	\emph{(a)} Assume that
	\[
	\begin{cases}
		\overline{u}_t \ge d_1\overline{u}_{xx} + (A(x - ct) - b_1\overline{u} - c_1 \underline{v}) \overline{u}, & 0 < t \le T,~ 0 < x < \overline{h}(t), \\
		\underline{v}_t \le d_2\underline{v}_{xx} + (a_2 - b_2\overline{u} - c_2\underline{v})\underline{v}, & 0 < t \le T,~ 0 < x < +\infty, \\
		\overline{u}(t, x) = 0, & 0 < t \le T, ~ \overline{h}(t) \leq x < +\infty, \\
		\overline{h}'(t) \ge -\mu (A(\overline{h}(t) - ct))\overline{u}_x(t, \overline{h}(t)), & 0 < t \le T, \\
		\overline{h}(0) \ge h_0,~ \overline{u}(0, x) \ge u_0(x), & 0 \leq x \leq h_0, \\
		\underline{v}(0, x) \le v_0(x), & 0 \le x < +\infty,
	\end{cases}
	\]
	and that, in addition, either
	\[
	\overline{u}_x(t, 0) \le 0 \le \underline{v}_x(t, 0)\quad (0<t\le T)
	\qquad\text{or}\qquad
	\overline{u}(t,0)\ge u(t,0),\ \ \underline v(t,0)\le v(t,0)\quad (0<t\le T).
	\]
	Then
	\[
	h(t) \le \overline{h}(t) \text{ in } (0,T],\quad u(t,x) \le \overline{u} (t,x), ~ v(t,x) \ge \underline{v} (t,x) \text{ for } (t,x) \in (0,T] \times [0,+\infty).
	\]

	\emph{(b)} Assume that
	\[
	\begin{cases}
		\underline{u}_t \le d_1\underline{u}_{xx} + (A(x - ct) - b_1\underline{u} - c_1 \overline{v}) \underline{u}, & 0 < t \le T,~ 0 < x < \underline{h}(t), \\
		\overline{v}_t \ge d_2\overline{v}_{xx} + (a_2 - b_2\underline{u} - c_2\overline{v})\overline{v}, & 0 < t \le T,~ 0 < x < +\infty, \\
		\underline{u}(t, x) = 0, & 0 < t \le T, ~ \underline{h}(t) \leq x < +\infty, \\
		\underline{h}'(t) \le -\mu (A(\underline{h}(t) - ct))\underline{u}_x(t, \underline{h}(t)), & 0 < t \le T, \\
		\underline{h}(0) \le h_0,~ \underline{u}(0, x) \le u_0(x), & 0 \leq x \leq \underline h(0), \\
		\overline{v}(0, x) \ge v_0(x), & 0 \le x < +\infty,
	\end{cases}
	\]
	and that, in addition, either
	\[
	\underline{u}_x(t, 0) \ge 0 \ge \overline{v}_x(t, 0)\quad (0<t\le T)
	\qquad\text{or}\qquad
	\underline{u}(t,0)\le u(t,0),\ \ \overline v(t,0)\ge v(t,0)\quad (0<t\le T).
	\]
	Then
	\[
	h(t) \ge \underline{h}(t) \text{ in } (0,T],\quad u(t,x) \ge \underline{u} (t,x), ~ v(t,x) \le \overline{v} (t,x) \text{ for } (t,x) \in (0,T] \times [0,+\infty).
	\]
\end{theorem}

\begin{proof}
	We prove (a); part (b) is obtained by exchanging the roles of the two triples.
	Let $\tilde{M}$ be a common upper bound of $v$ and $\underline{v}$ on
	$[0,T]\times[0,+\infty)$ and set $w:=\tilde{M}-v$, $\overline{w} :=
	\tilde{M}-\underline{v}$. Since $s\mapsto \tilde M-s$ reverses the order and
	the resulting system for $(u,w)$ is cooperative, it suffices to prove that
	$u\le\overline u$ and $w\le\overline w$.

	\textit{Step 1: the case $h_0<\overline h(0)$.} We claim that
	$h(t)<\overline{h}(t)$ for all $t\in(0,T]$. If not, there is a first
	$t^*\in(0,T]$ with $h(t)<\overline{h}(t)$ on $(0,t^*)$ and
	$h(t^*)=\overline{h}(t^*)$; consequently
	\begin{equation}\label{eq:h}
		h'(t^*) \ge \overline{h}'(t^*).
	\end{equation}
	Set $U := (\overline{u} - u)e^{-Kt}$ and $W := (\overline{w} - w)e^{-Kt}$. By
	the mean value theorem,
	\begin{equation}\label{eq:sup2}
		\begin{cases}
			U_t - d_1 U_{xx} \ge -KU + c_{11}U + c_{12}W, & 0 < t \le t^*, ~ 0 \le x < h(t), \\
			W_t - d_2 W_{xx} \ge -KW + c_{21}U + c_{22}W, & 0 < t \le t^*, ~ 0 \le x < \infty, \\
			U(t,x) \ge 0, & 0 < t \le t^*, ~ h(t) \le x < \infty, \\
			U(0,x) \ge 0, ~ W(0,x) \ge 0, & 0 \le x < \infty,
		\end{cases}
	\end{equation}
	where
	\begin{align*}
		c_{11} &= A(x-ct) - b_1 (u+\overline u) - c_1 (\tilde{M} - \eta_1), & c_{12} &= c_1 \xi_1, \\
		c_{21} &= b_2 (\tilde{M} - \eta_2), & c_{22} &= a_2 - b_2 \xi_2 - c_2\big((\tilde M-\eta_2)+(\tilde M-\eta_2')\big),
	\end{align*}
	with $\xi_i$ between $u$ and $\overline{u}$ and $\tilde{M}-\eta_i$,
	$\tilde M-\eta_i'$ between $v$ and $\underline{v}$. All the $c_{ij}$ are
	bounded on $[0,t^*]\times[0,\infty)$ by \cref{th6} and the boundedness of the
	comparison functions, and, crucially,
	\begin{equation}\label{eq:coop}
		c_{12}\ge0,\qquad c_{21}\ge0 .
	\end{equation}
	We fix $K\ge 1+|c_{11}|+|c_{12}|+|c_{21}|+|c_{22}|$ on
	$[0,t^*]\times[0,\infty)$. Note that the third line of \eqref{eq:sup2} holds
	because $u\equiv0$ and $\overline u\ge0$ for $x\ge h(t)$; that $\overline u\ge0$
	follows from the maximum principle applied on
	$\{0\le x\le\overline h(t),\,0\le t\le T\}$.

	For $\ell>\max_{0\le t\le t^*} h(t)$ define, with
	$d:=\max\{d_1,d_2\}$,
	\[
	\overline{U}(t,x) = U(t,x) + \frac{\tilde{M}(x^2 + 2dt)}{\ell^2},
	\qquad
	\overline{W}(t,x) = W(t,x) + \frac{\tilde{M}(x^2 + 2dt)}{\ell^2}.
	\]
	Using $\partial_t\frac{\tilde M(x^2+2dt)}{\ell^2}-d_i\partial_{xx}
	\frac{\tilde M(x^2+2dt)}{\ell^2}=\frac{2\tilde M(d-d_i)}{\ell^2}\ge0$,
	the pair $(\overline{U},\overline{W})$ satisfies on $[0,t^*]\times[0,\ell]$
	the same differential inequalities as $(U,W)$ up to a nonnegative right hand
	side, together with
	\[
	\overline{U}(t,x) \ge \tfrac{\tilde{M}(x^2 + 2dt)}{\ell^2} > 0 \ \ (h(t) \le x \le \ell),\qquad
	\overline{W}(t,\ell) \ge W(t,\ell) + \tilde{M} > 0,\qquad
	\overline{U}(0,\cdot),\overline{W}(0,\cdot)\ge0 ,
	\]
	the last inequality on $\{x=\ell\}$ because $|W|\le\tilde M$. Moreover, in the
	first alternative at $x=0$ we have
	$\overline U_x(t,0)\le0$ and $\overline W_x(t,0)\le0$, and in the second we
	have $\overline U(t,0)\ge0$ and $\overline W(t,0)\ge0$; in either case a
	negative minimum of $\overline U$ or $\overline W$ over
	$[0,t^*]\times[0,\ell]$ cannot be attained at $x=0$.

	Set
	$\kappa := \min\{\min_{[0,t^*]\times[0,\ell]}\overline{U},\,
	\min_{[0,t^*]\times[0,\ell]}\overline{W}\}$ and suppose $\kappa<0$. By the
	above, the minimum is attained either at some $(t_0,x_0)$ with
	$0<t_0\le t^*$, $0\le x_0<h(t_0)$ and $\overline U(t_0,x_0)=\kappa$, or at some
	$(t_1,x_1)$ with $0<t_1\le t^*$, $0\le x_1<\ell$ and
	$\overline W(t_1,x_1)=\kappa$. In the first case
	$\overline{U}_t(t_0,x_0)-d_1\overline{U}_{xx}(t_0,x_0)\le0$, whereas by
	\eqref{eq:coop} and the choice of $K$,
	\[
	-K\overline{U} + c_{11} \overline{U} + c_{12}\overline{W}
	\ \ge\ (-K + |c_{11}|)\kappa + c_{12}\kappa \ \ge\ -\kappa > 0
	\quad\text{ at }(t_0,x_0),
	\]
	a contradiction. The second case is identical. Hence $\kappa\ge0$, that is,
	\[
	U(t,x) \ge -\frac{\tilde{M}(x^2 + 2dt)}{\ell^2},
	\qquad W(t,x) \ge -\frac{\tilde{M}(x^2 + 2dt)}{\ell^2}
	\qquad (0\le t\le t^*,\ 0\le x\le\ell),
	\]
	and letting $\ell\to\infty$ gives $U\ge0$, $W\ge0$ on
	$[0,t^*]\times[0,\infty)$, i.e. $u\le\overline{u}$ and $v\ge\underline{v}$
	there.

	Finally, $Z:=\overline{u}-u$ satisfies
	$Z_t-d_1Z_{xx}\ge c_{11}Z+c_{12}(\overline{w}-w)\ge c_{11}Z$ on
	$\Omega_{t^*}:=\{0<t\le t^*,\ 0\le x<h(t)\}$ and $Z\ge0$, $Z\not\equiv0$;
	the strong maximum principle gives $Z>0$ in $\Omega_{t^*}$ and, since
	$Z(t^*,h(t^*))=\overline u(t^*,\overline h(t^*))=0$, the Hopf boundary lemma
	gives $Z_x(t^*,h(t^*))<0$, i.e. $\overline u_x(t^*,\overline h(t^*))<u_x(t^*,h(t^*))$.
	Consequently
	\[
	\overline h'(t^*)\ge-\mu\big(A(\overline h(t^*)-ct^*)\big)\overline u_x(t^*,\overline h(t^*))
	>-\mu\big(A(h(t^*)-ct^*)\big)u_x(t^*,h(t^*))=h'(t^*),
	\]
	contradicting \eqref{eq:h}. Therefore $h<\overline{h}$ on $(0,T]$, and the
	above argument applied on $[0,T]$ gives the assertion.

	\textit{Step 2: the case $h_0=\overline{h}(0)$.} Approximate the initial
	triple from below by compatible data
	$(u_{0,\epsilon},v_0,h_{0,\epsilon})$ with
	$h_{0,\epsilon}<h_0$, $u_{0,\epsilon}\le u_0$, and
	$h_{0,\epsilon}\uparrow h_0$, chosen so that, after the usual straightening
	of $[0,h_{0,\epsilon}]$ onto the fixed interval $[0,h_0]$, the data converge
	to $(u_0,v_0,h_0)$ in the norms used in the local-existence argument of
	\cref{th4}; such a compatible approximation is obtained, for instance, by a
	standard smooth cutoff near the moving endpoint. Let
	$(u_\epsilon,v_\epsilon,h_\epsilon)$ denote the corresponding solution.
	Step~1 gives $u_\epsilon\le\overline u$,
	$v_\epsilon\ge\underline v$ and $h_\epsilon<\overline h$.
	Continuous dependence after boundary straightening (proved together with
	\cref{th4} in \cref{sec_app}) yields convergence to $(u,v,h)$ on every
	compact time interval, and the conclusion follows by passing to the limit.
\end{proof}

\begin{remark}\label{rm:sub_super}
	$(\overline{u},\underline{v},\overline{h})$ and
	$(\underline{u},\overline{v},\underline{h})$ as in \cref{th7} are called an
	upper solution and a lower solution of \eqref{main}, respectively. The same
	proof applies verbatim when $A(\cdot-ct)$ and $\mu(A(\cdot-ct))$ are replaced
	by the constants $a_1$ and $\mu(a_1)$, i.e. \cref{th7} also holds for
	\eqref{main2}; and it applies when the two triples are the solutions of
	\eqref{main} corresponding to two ordered sets of data, which yields the
	monotone dependence of $(u,-v,h)$ on $(u_0,-v_0,h_0)$ and on $\mu$.
\end{remark}

\section{Long-time dynamics: proof of \texorpdfstring{\cref{th:dynamics}}{the dynamics theorem}}\label{sec_3}

By \cref{th6} the free boundary $h$ is strictly increasing, so
\[
h_\infty := \lim_{t\to+\infty}h(t)\in(h_0,+\infty]
\]
is well defined. We treat the two cases $h_\infty<+\infty$ and $h_\infty=+\infty$
separately.

\subsection{The vanishing case}

The following elementary observation is used repeatedly.

\begin{lemma}\label{lem:reduction}
	Suppose that there exists $T_0>0$ such that $h(t)<ct$ for every $t\ge T_0$;
	this is the case, with $T_0:=h_\infty/c$, whenever $h_\infty<+\infty$. Then
	for every $t\ge T_0$,
	\[
	A(x-ct)=a_1 \ \text{ for all } x\in[0,h(t)]\qquad\text{and}\qquad
	\mu(A(h(t)-ct))=\mu(a_1),
	\]
	and consequently $(t,x)\mapsto(u(t+T_0,x),v(t+T_0,x),h(t+T_0))$ is the
	solution of the homogeneous problem \eqref{main2} with $(u_0,v_0,h_0)$
	replaced by $(u(T_0,\cdot),v(T_0,\cdot),h(T_0))$.
\end{lemma}

\begin{proof}
	If $h_\infty<+\infty$ and $t\ge h_\infty/c$, then strict
	monotonicity of $h$ and the finiteness of its limit give $h(t)<h_\infty\le ct$.
	In either case, $x-ct\le h(t)-ct<0$ for every $x\in[0,h(t)]$, so \eqref{A} gives
	$A(x-ct)=a_1$ and $\mu(A(h(t)-ct))=\mu(a_1)$. Substituting into \eqref{main}
	yields \eqref{main2}.
\end{proof}

\begin{lemma}\label{th1:lm1}
	Let $(u,v,h)$ be the solution of \eqref{main} with initial functions
	satisfying \eqref{initial} and assume \eqref{weak_competition}. If
	$h_\infty<+\infty$, then
	\[
	h_\infty\le R^*,\qquad
	\lim_{t\to\infty} \|u(t,\cdot)\|_{C([0, h(t)])} = 0,\qquad
	\lim_{t\to\infty} v(t,\cdot) = \frac{a_2}{c_2}\ \text{ in } C^2_{\rm loc}([0,\infty)).
	\]
\end{lemma}

\begin{proof}
	By \cref{lem:reduction}, after the time shift $T_0=h_\infty/c$ the triple
	$(u,v,h)$ solves the homogeneous problem \eqref{main2}, whose initial data
	$(u(T_0,\cdot),v(T_0,\cdot),h(T_0))$ satisfy \eqref{initial}: indeed
	$u(T_0,\cdot)>0$ on $[0,h(T_0))$, $u_x(T_0,0)=u(T_0,h(T_0))=0$, $v(T_0,\cdot)>0$,
	$v_x(T_0,0)=0$, and $\liminf_{x\to\infty}v(T_0,x)>0$ by \cref{lem:vlow} below.
	Since for the shifted triple the free boundary still converges to
	$h_\infty<+\infty$, alternative (ii) of \cite[Theorem 1.1]{MR3986328} occurs
	for \eqref{main2}, that is,
	\[
	\lim_{t\to\infty}\|u(t,\cdot)\|_{C([0,h(t)])}=0,\qquad
	\lim_{t\to\infty}v(t,\cdot)=\frac{a_2}{c_2}\ \text{ in }C^2_{\rm loc}([0,\infty)).
	\]
	It remains to prove $h_\infty\le R^*$. Suppose, by contradiction, that
	$h_\infty>R^*$. Since $h$ is increasing with limit $h_\infty$, we may pick
	$T_1>T_0$ with $h(T_1)>R^*$. Applying \cref{lem:reduction} with initial time
	$T_1$, the triple $(u(\cdot+T_1,\cdot),v(\cdot+T_1,\cdot),h(\cdot+T_1))$
	solves \eqref{main2} with initial free boundary $h(T_1)>R^*$. By
	\cite[Theorem 1.2]{MR3986328}, spreading always happens for \eqref{main2} when
	the initial habitat exceeds
	$\frac{\pi}{2}\sqrt{d_1c_2/(a_1c_2-a_2c_1)}=R^*$, whatever the value of
	$\mu(a_1)>0$. Hence $h(t)\to+\infty$, contradicting $h_\infty<+\infty$.
\end{proof}

The following simple lemma, used above and again in \cref{sec_6}, records the
fact that the native population never dies out at infinity.

\begin{lemma}\label{lem:vlow}
	Let $(u,v,h)$ be the solution of \eqref{main} with \eqref{initial}. Let
	$(\hat u,\hat v)$ solve the kinetic system
	\begin{equation}\label{ode}
		\hat u' = (a_1 - b_1\hat u - c_1 \hat v )\hat u,\quad
		\hat v' = (a_2 -b_2 \hat u - c_2 \hat v)\hat v,\quad
		(\hat{u}(0), \hat{v}(0)) = \big(\|u_0\|_\infty, m_0\big),
	\end{equation}
	with $m_0$ as in \eqref{v0inf}. Then
	\[
	u(t,x)\le\hat u(t),\qquad v(t,x)\ge\hat v(t)\qquad (t>0,\ x\ge0),
	\]
	and $(\hat u(t),\hat v(t))\to(u^*,v^*)$ as $t\to\infty$. In particular, for
	every $\varepsilon>0$ there is $T_\varepsilon>0$ with
	\begin{equation}\label{eq:uv_global}
		u(t,x)\le u^*+\varepsilon,\qquad v(t,x)\ge v^*-\varepsilon
		\qquad (t\ge T_\varepsilon,\ x\ge0).
	\end{equation}
\end{lemma}

\begin{proof}
	Extend $u(t,\cdot)$ by $0$ on $[h(t),+\infty)$ and put
	\[
	U(t,x):=\hat u(t)-u(t,x),\qquad V(t,x):=v(t,x)-\hat v(t).
	\]
	On $0<x<h(t)$ a direct subtraction gives
	\[
	U_t-d_1U_{xx}
	=\big[a_1-b_1(\hat u+u)-c_1\hat v\big]U+c_1uV+(a_1-A(x-ct))u,
	\]
	where the last term is nonnegative. On the whole half-line,
	\[
	V_t-d_2V_{xx}
	=b_2\hat v\,U+\big[a_2-b_2u-c_2(v+\hat v)\big]V .
	\]
	Thus $(U,V)$ satisfies a linear cooperative parabolic system with bounded
	coefficients and a nonnegative source term. Moreover
	\[
	U(0,x)\ge0,\quad V(0,x)\ge0,\quad U_x(t,0)=V_x(t,0)=0,
	\quad U(t,h(t))=\hat u(t)>0.
	\]
	The maximum principle for cooperative systems therefore yields $U,V\ge0$.
	For completeness, the half-line causes no difficulty: one may apply the
	argument first on $[0,R]$ after adding the standard positive barrier
	$\varepsilon e^{Kt}(1+x^2)$ to both components, and then let
	$R\to\infty$ and $\varepsilon\downarrow0$. Hence
	$u(t,x)\le\hat u(t)$ and $v(t,x)\ge\hat v(t)$.

	Under \eqref{weak_competition} the coexistence equilibrium $(u^*,v^*)$ is
	globally asymptotically stable for the positive kinetic system; see, for
	example, \cite[Chapter 8]{MR1319817}. Consequently
	$(\hat u(t),\hat v(t))\to(u^*,v^*)$, and \eqref{eq:uv_global} follows.
\end{proof}

\subsection{The spreading case}

\begin{lemma}\label{th1:lm2}
	Let $(u,v,h)$ be the solution of \eqref{main} with \eqref{initial} and assume
	\eqref{weak_competition}. If $h_\infty=+\infty$, then
	\[
	\lim_{t\to\infty} (u(t,\cdot), v(t,\cdot)) = (u^*, v^*)\ \text{ in } C^2_{\rm loc}([0, \infty)).
	\]
\end{lemma}

\begin{proof}
	\textit{Step 0: upper bound for $u$ and lower bound for $v$.} Since
	$A(\cdot)\le a_1$ and $\mu(A(\cdot))\le\mu(a_1)$ while $-u_x(t,h(t))>0$, the
	triple $(u,v,h)$ is a lower solution of \eqref{main2} in the sense of
	\cref{th7}(b) (first alternative at $x=0$). Let
	$(\overline{u},\underline{v},\overline{h})$ denote the solution of
	\eqref{main2} with the same initial data. Then
	\begin{equation}\label{eq:comp_main2}
		u\le\overline{u},\qquad v\ge\underline{v},\qquad h\le\overline{h}.
	\end{equation}
	Since $h_\infty=+\infty$ we get $\overline{h}(t)\to+\infty$, so alternative
	(i) of \cite[Theorem 1.1]{MR3986328} holds for \eqref{main2}, that is,
	$(\overline{u},\underline{v})\to(u^*,v^*)$ in $C^2_{\rm loc}([0,\infty))$.
	Combining this with \eqref{eq:comp_main2},
	\begin{equation}\label{uv1}
		\limsup_{t \to +\infty} u(t, x) \le u^*, \qquad \liminf_{t \to +\infty} v(t, x) \ge v^*
	\end{equation}
	uniformly on compact subsets of $[0,\infty)$. It therefore suffices to prove
	\begin{equation}\label{uv2}
		\liminf_{t \to +\infty} u(t, x) \ge u^*, \qquad \limsup_{t \to +\infty} v(t, x) \le v^*
	\end{equation}
	locally uniformly in $[0,\infty)$.

	\textit{Step 1: an iteration scheme.} Set
	\[
	\overline{v}_0:=\frac{a_2}{c_2},\qquad
	\underline{u}_{n+1}:=\frac{a_1-c_1\overline{v}_n}{b_1},\qquad
	\overline{v}_{n+1}:=\frac{a_2-b_2\underline{u}_{n+1}}{c_2}\qquad(n\ge0).
	\]
	The map $\underline u\mapsto \frac{a_1}{b_1}-\frac{c_1}{b_1}
	\big(\frac{a_2}{c_2}-\frac{b_2}{c_2}\underline u\big)
	=\frac{a_1c_2-a_2c_1}{b_1c_2}+\frac{b_2c_1}{b_1c_2}\underline u$ is affine with
	slope $\frac{b_2c_1}{b_1c_2}\in(0,1)$, because \eqref{weak_competition} gives
	$c_1/c_2<b_1/b_2$; its unique fixed point is
	$\frac{a_1c_2-a_2c_1}{b_1c_2-b_2c_1}=u^*$. Since
	$\underline{u}_1=\frac{a_1c_2-a_2c_1}{b_1c_2}\in(0,u^*)$ by
	\eqref{ineq_star}, it follows that $\{\underline{u}_n\}$ is strictly
	increasing with $\underline{u}_n\uparrow u^*$ and $\{\overline{v}_n\}$ is
	strictly decreasing with $\overline{v}_n\downarrow v^*$; in particular
	\begin{equation}\label{eq:pos_iter}
		0<\underline{u}_n<u^*\ \ (n\ge1),
		\qquad
		v^*<\overline{v}_n\le\frac{a_2}{c_2}
		\ \text{ and }\
		a_1-c_1\overline{v}_n\ge a_1-\frac{a_2c_1}{c_2}>0\ \ (n\ge0),
	\end{equation}
	the last inequality by \eqref{ineq_star}.
	We claim that for every $n\ge0$,
	\begin{equation}\label{induc}
		\liminf_{t \to +\infty} u(t,x) \ge \underline{u}_n,  \qquad \limsup_{t \to +\infty} v(t,x) \le \overline{v}_n
	\end{equation}
	uniformly on compact subsets of $[0,\infty)$, with the convention
	$\underline u_0:=0$. Granting \eqref{induc}, letting $n\to\infty$ gives
	\eqref{uv2} and hence, together with \eqref{uv1}, the locally uniform
	convergence $(u,v)\to(u^*,v^*)$. Interior and Neumann-boundary parabolic estimates (applied on
	$[T,T+2]\times[0,\ell]$ for $\ell$ fixed and $T\to\infty$, using the uniform
	bounds of \cref{th6}) then upgrade this convergence to $C^2_{\rm loc}([0,\infty))$,
	which is the assertion of the lemma.

	\textit{Step 2: \eqref{induc} for $n=0$.} The inequality
	$\liminf u\ge\underline u_0=0$ is trivial. For $v$, the comparison principle
	applied to $v_t-d_2v_{xx}\le v(a_2-c_2v)$ and to the solution $V_1$ of
	$V_1'=V_1(a_2-c_2V_1)$, $V_1(0)=\|v_0\|_\infty$, gives $v(t,x)\le V_1(t)$ for
	all $x\ge0$, $t>0$, whence
	\begin{equation}\label{v1}
		\limsup_{t \to +\infty}{v(t,x)} \le \frac{a_2}{c_2} = \overline{v}_0
		\quad\text{ uniformly for } x \in [0,\infty).
	\end{equation}

	\textit{Step 3: from $n$ to $n+1$.} Assume \eqref{induc} holds for some
	$n\ge0$. Fix $\varepsilon>0$ so small that
	\begin{equation}\label{eq:eps_small}
		a_1-c_1(\overline v_n+\varepsilon)>0
		\qquad\text{and}\qquad
		a_2-b_2(\underline u_{n+1}-\varepsilon)>0,
	\end{equation}
	which is possible by \eqref{eq:pos_iter}, and fix
	\begin{equation}\label{eq:ell_large}
		\ell>\max\left\{h_0,\ \frac{\pi}{2}\sqrt{\frac{d_1}{a_1-c_1(\overline v_n+\varepsilon)}},\
		\frac{\pi}{2}\sqrt{\frac{d_2}{a_2-b_2(\underline u_{n+1}-\varepsilon)}}\right\}.
	\end{equation}
	Since $h_\infty=+\infty$ and by the induction hypothesis, there exists
	$t_n>0$ such that
	\begin{equation}\label{eq:tn}
		h(t)>\ell,\quad ct>\ell,\quad
		u(t,x)\ge\underline{u}_n-\varepsilon,\quad
		v(t,x)\le \overline{v}_n+\varepsilon
		\qquad (t>t_n,\ x\in[0,\ell]).
	\end{equation}
	Because $ct>\ell$ we have $A(x-ct)=a_1$ for all $x\in[0,\ell]$ and $t>t_n$, so
	$u$ satisfies
	\[
	\begin{cases}
		u_t - d_1u_{xx} \ge \big(a_1 - b_1 u - c_1 (\overline{v}_n+\varepsilon)\big) u , & 0 < x< \ell, ~ t > t_n,\\
		u_x(t, 0) = 0, \quad u(t, \ell) > 0, & t > t_n .
	\end{cases}
	\]
	Choose $\eta\in C^2([0,\ell])$, $\eta\not\equiv0$, with
	$\eta\ge0$, $\eta'(0)=\eta(\ell)=0$ and $\eta(x)\le u(t_n,x)$ on
	$[0,\ell]$ (for instance a sufficiently small multiple of
	$\cos(\pi x/(2\ell))$). Let $\hat u$ solve
	\[
	\begin{cases}
		\hat{u}_t = d_1 \hat{u}_{xx} + \big(a_1 - b_1 \hat{u} - c_1 (\overline{v}_n + \varepsilon)\big)\hat{u}, & 0 < x < \ell, ~ t > t_n,\\
		\hat{u}_x(t,0) = 0, \quad \hat{u}(t,\ell) = 0, & t > t_n,\\
		\hat{u}(t_n,x) = \eta(x), & 0 \le x \le \ell.
	\end{cases}
	\]
	Then the scalar comparison principle gives $u(t,x)\ge\hat u(t,x)$ on
	$[0,\ell]$ for $t>t_n$. By the choice
	\eqref{eq:ell_large} of $\ell$, the principal eigenvalue of
	$-d_1\partial_{xx}$ on $(0,\ell)$ with Neumann condition at $0$ and Dirichlet
	condition at $\ell$, namely $d_1\pi^2/(4\ell^2)$, is smaller than $a_1-c_1(\overline v_n+\varepsilon)$. Hence the
	above logistic problem has a unique positive steady state $\tilde u_\ell$,
	and $\hat u(t,\cdot)\to\tilde u_\ell$ uniformly on $[0,\ell]$ as
	$t\to\infty$. Moreover
	$\tilde{u}_\ell\to\frac{a_1-c_1(\overline v_n+\varepsilon)}{b_1}$ locally
	uniformly on $[0,\infty)$ as $\ell\to\infty$, which is the standard
	large-domain limit for the logistic Dirichlet--Neumann problem (see, e.g.,
	\cite{DU_MA_2001}). Since $\ell$ may be taken arbitrarily large and
	$\varepsilon>0$ arbitrarily small, we conclude
	\begin{equation}\label{eq:un1}
		\liminf_{t \to +\infty}{u(t,x)} \ge \frac{a_1-c_1\overline{v}_n}{b_1} = \underline{u}_{n+1}
		\quad\text{ locally uniformly in }[0,\infty).
	\end{equation}
	The bound for $v$ is obtained in exactly the same way. Indeed, by
	\eqref{eq:un1} there is $t_n'>t_n$ with
	$u(t,x)\ge\underline{u}_{n+1}-\varepsilon$ for $x\in[0,\ell]$, $t>t_n'$, and
	therefore
	\[
	\begin{cases}
		v_t - d_2v_{xx} \le \big(a_2 - b_2(\underline{u}_{n+1}-\varepsilon) - c_2 v\big) v , & 0 < x< \ell, ~ t > t_n',\\
		v_x(t, 0) = 0, \quad v(t, \ell) \le M_2, & t > t_n' ,
	\end{cases}
	\]
	so that $v\le\hat v$ on $[0,\ell]\times(t_n',\infty)$, where $\hat v$ solves
	the corresponding logistic problem with boundary value $M_2$ at $x=\ell$ and
	constant initial value $\hat v(t_n',x)=M_2$. This datum is compatible with the
	boundary conditions and dominates $v(t_n',\cdot)$. Again by
	\eqref{eq:ell_large} this problem has a unique positive steady state, to which
	$\hat v(t,\cdot)$ converges, and that
	steady state tends to $\frac{a_2-b_2(\underline u_{n+1}-\varepsilon)}{c_2}$
	locally uniformly as $\ell\to\infty$. Letting $\ell\to\infty$ and
	$\varepsilon\to0$,
	\[
	\limsup_{t \to +\infty}{v(t,x)} \le \frac{a_2-b_2\underline{u}_{n+1}}{c_2} = \overline{v}_{n+1}
	\quad\text{ locally uniformly in }[0,\infty),
	\]
	which completes the induction and the proof.
\end{proof}

\begin{proof}[Proof of \cref{th:dynamics}]
	By \cref{th6}, $h_\infty$ exists in $(h_0,+\infty]$. If $h_\infty<+\infty$,
	\cref{th1:lm1} gives alternative (ii); if $h_\infty=+\infty$,
	\cref{th1:lm2} gives alternative (i). The two alternatives are mutually
	exclusive.
\end{proof}

\section{Criteria for spreading and vanishing: proof of \texorpdfstring{\cref{th:criteria}}{the criteria theorem}}\label{sec_4}

\begin{proof}[Proof of \cref{th:criteria}(i)]
	Assume $h_0\ge R^*$ and suppose that vanishing occurs. By \cref{th1:lm1} we
	would have $h_\infty\le R^*\le h_0$, contradicting $h_\infty>h_0$. Hence
	spreading occurs.
\end{proof}

\begin{proof}[Proof of \cref{th:criteria}(ii)]
	Assume $h_0<R_0$ and write $u_0=\sigma\phi$.

	\textit{Monotonicity in $\sigma$.} If $0<\sigma_1<\sigma_2$ then
	$\sigma_1\phi\le\sigma_2\phi$, and by \cref{rm:sub_super} the triple
	$(u^{\sigma_1},v^{\sigma_1},h^{\sigma_1})$ is a lower solution of the problem
	solved by $(u^{\sigma_2},v^{\sigma_2},h^{\sigma_2})$; the comparison principle
	yields $h^{\sigma_1}(t)\le h^{\sigma_2}(t)$ for all $t\ge0$. Hence, if
	spreading occurs for $\sigma_1$, then it occurs for every $\sigma\ge\sigma_1$.
	Consequently the set
	$\Sigma:=\{\sigma>0:\ \text{vanishing occurs for } u_0=\sigma\phi\}$ is an
	interval of the form $(0,\sigma_0]$ or $(0,\sigma_0)$ with
	$\sigma_0\in[0,+\infty]$, and it remains to rule out $\sigma_0=0$ and to check
	that the endpoint belongs to $\Sigma$.

	\textit{$\sigma_0>0$.} Since $A(\cdot)\le a_1$ and $v>0$, the pair $(u,h)$
	satisfies
	\[
	u_t-d_1u_{xx}\le u(a_1-b_1u)\ \ \text{ in }\{0<x<h(t)\},\qquad
	h'(t)\le -\mu(a_1)u_x(t,h(t)),
	\]
	together with $u_x(t,0)=0$ and $u(t,h(t))=0$. Thus $(u,h)$ is a lower solution
	of the single species free boundary problem
	\begin{equation}\label{DL}
		\begin{cases}
			U_t = d_1U_{xx}+U(a_1-b_1U), & t>0,\ 0<x<H(t),\\
			U_x(t,0)=U(t,H(t))=0,\ \ H'(t)=-\mu(a_1)U_x(t,H(t)), & t>0,\\
			H(0)=h_0,\ U(0,x)=\sigma\phi(x), & 0\le x\le h_0,
		\end{cases}
	\end{equation}
	so $h(t)\le H(t)$ for all $t$. Assume
	$h_0<R_0=\frac{\pi}{2}\sqrt{d_1/a_1}$. The upper solution constructed in the
	proof of \cite[Lemma 3.8]{MR2607347} then applies to \eqref{DL}: it forces
	$H_\infty<+\infty$ as soon as $\mu(a_1)M$ is small enough, where $M$ is any
	constant with $\sigma\phi(x)\le M\cos\big(\frac{\pi x}{2h_0(1+\delta/2)}\big)$
	on $[0,h_0]$ and $\delta$ depends only on $d_1,a_1,h_0$. Since $M$ may be
	taken proportional to $\sigma\|\phi\|_{\infty}$, vanishing occurs for
	\eqref{DL} whenever $\sigma\|\phi\|_{\infty}$ is small enough, $\mu(a_1)$
	being fixed. In that case $H_\infty<+\infty$, so
	$h_\infty\le H_\infty<+\infty$; hence vanishing for \eqref{main} follows from
	\cref{th1:lm1}. Therefore $\sigma_0>0$.

	\textit{The endpoint.} Let $\sigma_0<+\infty$ and suppose spreading occurred
	for $\sigma=\sigma_0$. Then $h^{\sigma_0}(T)>R^*$ for some $T>0$. The solution
	of \eqref{main} depends continuously on its initial data on the fixed time
	interval $[0,T]$: this follows from \cref{th4} and a standard continuation
	argument, the solution being global with bounds that are uniform for
	$\sigma$ in a compact set by \cref{th6}. Hence there is
	$\sigma<\sigma_0$ with $h^{\sigma}(T)>R^*$, and
	\cref{th1:lm1} forces spreading for that $\sigma$, contradicting
	$\sigma\in\Sigma$. Hence $\sigma_0\in\Sigma$, which is the assertion.
\end{proof}

\section{Semi-wave solutions}\label{sec_5}

Throughout this section we assume \eqref{weak_competition}. We first recall the
two results of \cite{MR3986328} that we shall use.

\begin{theorem}[see {\cite[Proposition 1.4]{MR3986328}} and \cref{rm:convert}]\label{th:c*}
	There exists a critical speed
	\[
	c^*\ \ge\ 2\sqrt{d_1\left(a_1-\frac{a_2c_1}{c_2}\right)}\ >0
	\]
	such that the problem
	\begin{equation}\label{tw}
		\begin{cases}
			-c \Psi' - d_1 \Psi'' = (a_1 - b_1\Psi - c_1\Phi)\Psi, & -\infty < x < +\infty,\\
			-c \Phi' - d_2 \Phi'' = (a_2 - b_2\Psi - c_2\Phi)\Phi, & -\infty < x < +\infty,\\
			\Psi(-\infty) = u^*, ~ \Psi(+\infty)=0, ~ \Psi'(x)<0 \text{ for } x \in \mathbb{R},\\
			\Phi(-\infty) = v^*, ~ \Phi(+\infty)=\frac{a_2}{c_2}, ~ \Phi'(x)>0 \text{ for } x\in \mathbb{R}
		\end{cases}
	\end{equation}
	has a solution when $c\ge c^*$ and has no solution when $c<c^*$.
\end{theorem}

\begin{theorem}[see {\cite[Theorem 1.5]{MR3986328}} and \cref{rm:convert}]\label{th:semi-wave0}
	For each $c\in[0,c^*)$ the problem
	\begin{equation}\label{semiwave0}
		\begin{cases}
			-c \psi' - d_1 \psi'' = (a_1 - b_1\psi - c_1\varphi)\psi, & -\infty < x < 0,\\
			-c \varphi' - d_2 \varphi'' = (a_2 - b_2\psi - c_2\varphi)\varphi, & -\infty < x < +\infty,\\
			\psi(-\infty) = u^*, ~ \psi(x)=0 \text{ for } x\ge 0, ~ \psi'(x)<0 \text{ for } x\le 0,\\
			\varphi(-\infty) = v^*, ~ \varphi(+\infty)=\frac{a_2}{c_2}, ~ \varphi'(x)>0 \text{ for } x\in \mathbb{R}
		\end{cases}
	\end{equation}
	has a unique solution $(\Psi_c, \Phi_c) \in [C(\mathbb{R})\cap C^2((-\infty, 0])] \times C^2(\mathbb{R})$,
	and it has no such solution for $c\ge c^*$. The map
	$c\mapsto(\Psi_c,\Phi_c)$ is continuous from $[0,c^*)$ into
	$C^2_{\rm loc}((-\infty,0])\times C^2_{\rm loc}(\mathbb{R})$; in particular
	$c\mapsto\Psi_c'(0)$ is continuous on $[0,c^*)$. Moreover there is a unique
	$c_0\in(0,c^*)$ with
	\begin{equation}\label{c0}
		-\mu(a_1) \Psi_{c_0}'(0) = c_0 .
	\end{equation}
\end{theorem}

\begin{remark}\label{rm:convert}
	Theorems~\ref{th:c*} and \ref{th:semi-wave0} are stated in \cite{MR3986328}
	for the non-dimensional system obtained from \eqref{main2} by the scaling
	$\hat u(t,x)=\frac{b_1}{a_1}u\big(t/a_1,\sqrt{d_1/a_1}\,x\big)$,
	$\hat v(t,x)=\frac{c_2}{a_2}v\big(t/a_1,\sqrt{d_1/a_1}\,x\big)$, in which the
	competition coefficients become $k=\frac{a_2c_1}{a_1c_2}$ and
	$h=\frac{a_1b_2}{a_2b_1}$, so that \eqref{weak_competition} is equivalent to
	$k,h\in(0,1)$. Speeds are multiplied by $\sqrt{a_1d_1}$ under this scaling,
	so the bound $c^*\ge2\sqrt{1-k}$ of \cite[Proposition 1.4]{MR3986328} becomes
	$c^*\ge2\sqrt{a_1d_1(1-k)}=2\sqrt{d_1(a_1-a_2c_1/c_2)}$, as stated in
	\cref{th:c*}. The continuity of $c\mapsto(\Psi_c,\Phi_c)$ follows from the
	uniqueness statement together with standard elliptic estimates and a
	compactness argument, exactly as in \cite[Theorem 1.3(ii)]{MR3609207}.
We shall also use continuity of the selected speed $c_0$ under the small
coefficient perturbations introduced below. This parameter-continuity is not
stated explicitly in \cite{MR3986328}, but it follows from the same
compactness principle. We record the argument along the perturbation paths
used in \cref{rm:perturbed}. Let the perturbed coefficients converge to the
original ones, write $a_{1,n}$ and $a_{2,n}$ for the corresponding favourable
plateau and resident growth rate, and let $c_n$ be the selected homogeneous
speed.
Scalar comparison in the absence of the resident, as in \cref{lm:c0sing},
gives a uniform upper bound for $c_n$.

The speeds cannot converge to zero. Indeed, put
$r_n=a_{1,n}-c_1a_{2,n}/c_2$. Strict weak competition gives
$r_n\to r_0:=a_1-c_1a_2/c_2>0$. If a subsequence satisfied $c_n\to0$, then,
for all large $n$, the scalar semi-wave $q_n$ of
\[
-c_nq_n'-d_1q_n''=(r_n-b_1q_n)q_n,\qquad
q_n(-\infty)=r_n/b_1,\quad q_n(0)=0,
\]
would exist by the standard scalar semi-wave theory (see, e.g.,
\cite{MR2607347,HU20205931}). Since the resident component of the competition
semi-wave is at most $a_{2,n}/c_2$, its invader component is a supersolution of this scalar
problem and has the larger left state $u_n^*>r_n/b_1$. The one-dimensional
comparison principle and the Hopf lemma therefore give
$\Psi_n'(0)\le q_n'(0)<0$. The Stefan selection condition then yields
\[
c_n=-\mu(a_{1,n})\Psi_n'(0)
\ge -\mu(a_{1,n})q_n'(0).
\]
Standard continuous dependence of the scalar semi-wave at $c=0$ makes the
right-hand side converge to a strictly positive number, contradicting
$c_n\to0$.

Hence, after extraction, $c_n\to c_*>0$. Standard local elliptic estimates
give a locally convergent subsequence of the associated competition
semi-waves. Its limit solves the homogeneous semi-wave problem at speed
$c_*$ and satisfies $-\mu(a_1)\Psi'(0)=c_*$. It is nontrivial because
$\Psi'(0)=-c_*/\mu(a_1)<0$. Since the profiles are monotone, their limit at
$-\infty$ is a positive kinetic equilibrium. Nontriviality and monotonicity
exclude $(0,a_2/c_2)$, while the bound $\Psi_n\le u_n^*$ and
$u_n^*<a_{1,n}/b_1$ exclude $(a_1/b_1,0)$. Hence the left state is
$(u^*,v^*)$. Uniqueness in
\cref{th:semi-wave0} identifies $c_*=c_0$. Thus every convergent subsequence
has the same limit and, in particular, the perturbed speeds used below satisfy
$c_{0,\delta}^{\pm}\to c_0$ as $\delta\downarrow0$.
\end{remark}

The main result of this section supplies the semi-waves adapted to the shifting
environment. It is the analogue of \cite[Proposition 1.1]{HU20205931} for the
competition system, and of the corresponding result of \cite{DLN2026a} in the
weak--strong regime.

\begin{theorem}\label{th:semi-wave}
	Assume \eqref{weak_competition} and $0<c\le c_0$. For every $L\ge0$ there
	exists a pair
	\[
	(\psi_L,\varphi_L)\in
	[C(\mathbb R)\cap C^2((-\infty,L])]\times C^2(\mathbb R)
	\]
	satisfying
	\begin{equation}\label{eqlm1}
		\begin{cases}
			-c \psi_L' - d_1 \psi_L'' =
			(A(x)-b_1\psi_L-c_1\varphi_L)\psi_L,
			& -\infty<x<L,\\
			-c \varphi_L' - d_2 \varphi_L'' =
			(a_2-b_2\psi_L-c_2\varphi_L)\varphi_L,
			& -\infty<x<+\infty,\\
			\psi_L(-\infty)=u^*,\qquad \psi_L(x)=0,
			& x\ge L,\\
			\varphi_L(-\infty)=v^*,\qquad
			\varphi_L(+\infty)=\dfrac{a_2}{c_2}.
		\end{cases}
	\end{equation}
	Moreover
	\begin{equation}\label{eq:sw-range}
		0<\psi_L<u^*\quad\hbox{on }(-\infty,L),\qquad
		v^*<\varphi_L<\frac{a_2}{c_2}\quad\hbox{on }\mathbb R,
	\end{equation}
	$\psi_L$ is strictly decreasing on $(-\infty,L]$ and $\varphi_L$ is
	strictly increasing on $\mathbb R$. Consequently
	$\psi_L'\le0$, $\varphi_L'\ge0$, and the Hopf lemma gives
	$\psi_L'(L)<0$.

For any choice of semi-wave $(\psi_L,\varphi_L)$ furnished by the
construction below for each $L$, the following two endpoint properties hold:
	\begin{align}
		&\psi_L'(L)\longrightarrow0
		&&\text{as }L\to+\infty, \label{eq:sw-largeL}\\
		&\psi_L'(L)\longrightarrow\Psi_c'(0)
		&&\text{as }L\downarrow0. \label{eq:sw-smallL}
	\end{align}
	No uniqueness of the inhomogeneous semi-wave is asserted or needed below.
\end{theorem}

The proof occupies the rest of this section. It is convenient to transform
\eqref{eqlm1} into a cooperative system. Setting
\begin{equation}\label{cooper}
	\Theta_1:=\psi,\qquad
	\Theta_2:=\frac{a_2}{c_2}-\varphi ,
\end{equation}
we obtain
\begin{equation}\label{eqcoop}
	\begin{cases}
		d_1\Theta_1''+c\Theta_1'+F_1(x,\Theta)=0,
		& -\infty<x<L,\\
		d_2\Theta_2''+c\Theta_2'+F_2(\Theta)=0,
		& -\infty<x<+\infty,
	\end{cases}
\end{equation}
where
\[
F_1(x,\Theta)
 :=\Big(A(x)-\frac{a_2c_1}{c_2}-b_1\Theta_1+c_1\Theta_2\Big)\Theta_1,
\qquad
F_2(\Theta)
 :=(b_2\Theta_1-c_2\Theta_2)
 \Big(\frac{a_2}{c_2}-\Theta_2\Big).
\]
On
\[
\mathcal R:=
\Big[0,u^*\Big]\times
\Big[0,\frac{a_2}{c_2}-v^*\Big]
\]
we have
\[
\partial_{\Theta_2}F_1=c_1\Theta_1\ge0,\qquad
\partial_{\Theta_1}F_2
=b_2\Big(\frac{a_2}{c_2}-\Theta_2\Big)\ge0,
\]
so the system is cooperative on $\mathcal R$.

\subsection{Existence and qualitative properties}

\begin{lemma}\label{lm:supersub}
	Let $0<c\le c_0$ and $L\ge0$. Then
	\[
	\overline\Theta:=
	\Big(u^*,\frac{a_2}{c_2}-v^*\Big),\qquad
	\underline\Theta:=
	\Big(\Psi_{c_0},\frac{a_2}{c_2}-\Phi_{c_0}\Big)
	\]
	are respectively an upper and a weak lower solution of \eqref{eqcoop},
	and $\underline\Theta\le\overline\Theta$ componentwise.
\end{lemma}

\begin{proof}
	The ordering follows from the monotonicity and endpoint values of
	$(\Psi_{c_0},\Phi_{c_0})$. Since
	$a_1-b_1u^*-c_1v^*=a_2-b_2u^*-c_2v^*=0$ and $A\le a_1$,
	\[
	F_1\Big(x,u^*,\frac{a_2}{c_2}-v^*\Big)
	=(A(x)-a_1)u^*\le0,\qquad
	F_2\Big(u^*,\frac{a_2}{c_2}-v^*\Big)=0,
	\]
	which gives the upper-solution inequalities.

	For $x<0$, using the homogeneous semi-wave equation at speed $c_0$,
	\[
	d_1\Psi_{c_0}''+c\Psi_{c_0}'
	+F_1(x,\underline\Theta)
	=(c-c_0)\Psi_{c_0}'\ge0,
	\]
	because $c\le c_0$ and $\Psi_{c_0}'<0$. On $(0,L)$ the first component
	of $\underline\Theta$ is zero. At $x=0$ it has the admissible upward corner
	$\Psi_{c_0}'(0^-)<0=\Psi_{c_0}'(0^+)$, hence is a weak lower solution in the
	usual monotone-iteration sense; compare
	\cite[Definition 2.4]{MR3986328}. Likewise,
	\[
	d_2\Big(\frac{a_2}{c_2}-\Phi_{c_0}\Big)''
	+c\Big(\frac{a_2}{c_2}-\Phi_{c_0}\Big)'
	+F_2(\underline\Theta)
	=-(c-c_0)\Phi_{c_0}'\ge0 .
	\]
\end{proof}

\begin{proof}[Proof of existence in \cref{th:semi-wave}]
	Fix $L\ge0$. On $(-n,L)\times(-n,n)$, with $n>L$, impose the boundary data
	\[
	\Theta_1(-n)=\overline\Theta_1,\quad \Theta_1(L)=0,\qquad
	\Theta_2(-n)=\overline\Theta_2,\quad \Theta_2(n)=\overline\Theta_2.
	\]
	In the second equation, $\Theta_1^n$ is extended by zero on $[L,n]$,
	consistently with the convention in \eqref{eqlm1}.
	The ordered barriers of \cref{lm:supersub} and the standard monotone
	iteration for cooperative elliptic systems yield a solution $\Theta^n$
	between them. Interior estimates, together with the one-dimensional boundary
	estimates at $x=L$, are uniform on compact subsets; hence a diagonal
	extraction gives a solution $(\psi_L,\varphi_L)$ of \eqref{eqlm1} satisfying
	\begin{equation}\label{s0}
		\Psi_{c_0}(x)\le\psi_L(x)\le u^*,\qquad
		v^*\le\varphi_L(x)\le\Phi_{c_0}(x)<\frac{a_2}{c_2}.
	\end{equation}
	In particular $\psi_L(-\infty)=u^*$ and
	$\varphi_L(-\infty)=v^*$.

	The strong comparison principle applied to the ordered solution and the
	constant upper state shows that the inequalities are strict:
	$0<\psi_L<u^*$ on $(-\infty,L)$ and $\varphi_L>v^*$ on $\mathbb R$.
	The boundary Hopf lemma gives $\psi_L'(L)<0$. The remaining limit at $+\infty$ is proved in \cref{lm:phi_infty} below,
	and the strict monotonicity in \cref{lm:mono}; neither uses anything beyond
	the bounds \eqref{s0} established above.
\end{proof}

\begin{lemma}\label{lm:phi_infty}
	For every semi-wave constructed above,
	$\varphi_L'>0$ on $(L,+\infty)$ and
	$\varphi_L(+\infty)=a_2/c_2$.
\end{lemma}

\begin{proof}
	For $x>L$ one has $\psi_L=0$, and therefore
	\[
	-c\varphi_L'-d_2\varphi_L''
	=(a_2-c_2\varphi_L)\varphi_L>0.
	\]
	Thus
	$\big(e^{cx/d_2}\varphi_L'\big)'<0$ on $(L,+\infty)$.
	If $\varphi_L'(x_1)\le0$ at some $x_1>L$, then
	$\varphi_L'<0$ for all $x>x_1$. The bounded function $\varphi_L$ would then
	have a limit $m\in(0,a_2/c_2)$, and standard ODE compactness followed by
	passage to the equation would give $(a_2-c_2m)m=0$, a contradiction.
	Hence $\varphi_L'>0$ on $(L,+\infty)$. Its limit $m$ therefore exists, and
	the same limiting argument yields $m=a_2/c_2$.
\end{proof}

\begin{lemma}\label{lm:asym}
	Let $(\psi_L,\varphi_L)$ be a semi-wave obtained above. Let $\lambda_1$ be
	the smaller positive root of
	\begin{equation}\label{charac}
	\Lambda(\lambda):=
	\big(d_1\lambda^2+c\lambda-b_1u^*\big)
	\big(d_2\lambda^2+c\lambda-c_2v^*\big)
	-b_2c_1u^*v^* .
	\end{equation}
	Then $\Lambda$ has exactly two positive roots
	$0<\lambda_1<\lambda_2$, and there are $k_1,k_2>0$ such that
	\begin{equation}\label{asym}
	u^*-\psi_L(x)=(k_1+o(1))e^{\lambda_1x},\qquad
	\varphi_L(x)-v^*=(k_2+o(1))e^{\lambda_1x}
	\quad (x\to-\infty).
	\end{equation}
\end{lemma}

\begin{proof}
	Put
	$p(\lambda)=d_1\lambda^2+c\lambda-b_1u^*$ and
	$q(\lambda)=d_2\lambda^2+c\lambda-c_2v^*$, and let
	$\alpha_-<0<\alpha_+$ be the roots of $p$. Since
	\[
	\Lambda(0)=u^*v^*(b_1c_2-b_2c_1)>0,\qquad
	\Lambda(\alpha_\pm)=-b_2c_1u^*v^*<0,
	\]
	and $\Lambda(\lambda)\to+\infty$ as $|\lambda|\to\infty$, the quartic
	$\Lambda$ has one zero in each of
	$(-\infty,\alpha_-)$, $(\alpha_-,0)$, $(0,\alpha_+)$ and
	$(\alpha_+,+\infty)$. Thus it has exactly two positive roots
	$0<\lambda_1<\alpha_+<\lambda_2$.

	Set
	$\widehat\psi=u^*-\psi_L>0$ and
	$\widehat\varphi=\varphi_L-v^*>0$. Since $A\equiv a_1$ on
	$(-\infty,0]$, the pair satisfies
	\begin{equation}\label{lin}
	\begin{cases}
	d_1\widehat\psi''+c\widehat\psi'
	-b_1u^*\widehat\psi+c_1u^*\widehat\varphi
	=-(b_1\widehat\psi-c_1\widehat\varphi)\widehat\psi,\\
	d_2\widehat\varphi''+c\widehat\varphi'
	+b_2v^*\widehat\psi-c_2v^*\widehat\varphi
	=-(b_2\widehat\psi-c_2\widehat\varphi)\widehat\varphi .
	\end{cases}
	\end{equation}
	The equilibrium of the associated four-dimensional first-order system is
	hyperbolic. Standard asymptotic-integration theory for a trajectory on its
	unstable manifold (see, e.g., \cite[Chapters IX--X]{MR658490}) shows that the leading term is an eigenmode
	$e^{\lambda_jx}\xi^{(j)}$, $j=1$ or $2$. An eigenvector at a positive root
	$\lambda$ obeys
	\[
	p(\lambda)\xi_1+c_1u^*\xi_2=0.
	\]
	For $\lambda_1<\alpha_+$, the two components have the same sign; for
	$\lambda_2>\alpha_+$, they have opposite signs. Positivity of
	$(\widehat\psi,\widehat\varphi)$ excludes the $\lambda_2$ mode, proving
	\eqref{asym} with $k_1,k_2>0$.
\end{proof}

\begin{lemma}\label{lm:mono}
	$\psi_L$ is strictly decreasing on $(-\infty,L]$ and
	$\varphi_L$ is strictly increasing on $\mathbb R$.
\end{lemma}

\begin{proof}
	For $s>0$ put
	$\psi^s(x)=\psi_L(x+s)$ and
	$\varphi^s(x)=\varphi_L(x+s)$. Since $A$ is nonincreasing,
	$A(x+s)\le A(x)$; in the cooperative variables \eqref{cooper}, the shifted
	pair is therefore a lower solution on $(-\infty,L-s)$.

	We give the sliding argument because it is also the point at which the
	behaviour at $-\infty$ matters. By \cref{lm:asym}, for each
	$\varepsilon\in(0,1)$ there is $X_\varepsilon<0$ such that
	\begin{equation}\label{eq:two_sided}
	\begin{aligned}
	(1-\varepsilon)k_1e^{\lambda_1x}
	&\le u^*-\psi_L(x)
	\le(1+\varepsilon)k_1e^{\lambda_1x},\\
	(1-\varepsilon)k_2e^{\lambda_1x}
	&\le\varphi_L(x)-v^*
	\le(1+\varepsilon)k_2e^{\lambda_1x}
	\end{aligned}
	\qquad (x\le X_\varepsilon).
	\end{equation}
	We first show that the slide can be started for all sufficiently large shifts.
	Write $y=x+\sigma$. Choose $X<X_\varepsilon$. If $y\le X$, then both
	$y$ and $y-\sigma$ lie in the asymptotic region, and \eqref{eq:two_sided}
	gives
	\[
	 u^*-\psi_L(y) > u^*-\psi_L(y-\sigma),\qquad
	 \varphi_L(y)-v^*>\varphi_L(y-\sigma)-v^*
	\]
	for every sufficiently large $\sigma$, uniformly in $y\le X$. On the
	compact interval $[X,L]$ we have $\psi_L<u^*$ and $\varphi_L>v^*$;
	hence, as $\sigma\to\infty$, the convergence
	$(\psi_L(y-\sigma),\varphi_L(y-\sigma))\to(u^*,v^*)$ is uniform for
	$y\in[X,L]$, and the same two inequalities hold there for all large
	$\sigma$. This proves the required ordering for the first component.
	For the second component on the remaining right-hand tail, note that
	$\varphi_L$ is already strictly increasing on $(L,+\infty)$ by
	\cref{lm:phi_infty}. If $y-\sigma\ge L$ the desired inequality follows
	directly. If $y-\sigma<L$, then
	$\varphi_L(y-\sigma)\le M_L:=\sup_{z\le L}\varphi_L(z)<a_2/c_2$;
	choosing $Y>L$ so large that $\varphi_L(y)>M_L$ for $y\ge Y$, and treating
	the compact interval $[L,Y]$ as above, completes the start of the slide.

	More explicitly, let
	\[
	S:=\{s>0:\ \psi^\sigma\le\psi_L\text{ on }(-\infty,L-\sigma],
	\ \varphi^\sigma\ge\varphi_L\text{ on }\mathbb R
	\text{ for every }\sigma\ge s\}.
	\]
	The preceding tail estimates, followed by compactness on the remaining
	bounded intervals, show that $S\ne\emptyset$. Let
	$\bar s=\inf S$. If $\bar s>0$, the differences
	\[
	P=\psi_L-\psi^{\bar s},\qquad
	Q=\varphi^{\bar s}-\varphi_L
	\]
	satisfy a linear cooperative system with bounded coefficients,
	nonnegative off-diagonal entries, and
$P(L-\bar s)=\psi_L(L-\bar s)>0$. The strict tail inequalities and the
strong maximum principle give $P>0$ and $Q>0$. First note that $S$ is closed
under decreasing limits, so $\bar s\in S$. On each bounded middle interval
the strict differences at the shift $\bar s$ are bounded away from zero and
hence persist for all shifts sufficiently close to $\bar s$. For the left
tail choose $\varepsilon$ so small that
$\frac{1-\varepsilon}{1+\varepsilon}>e^{-\lambda_1\bar s/2}$.
Then the two-sided estimates \eqref{eq:two_sided} give the required strict
ordering uniformly for shifts in a fixed neighbourhood of $\bar s$ contained
in $[\bar s/2,+\infty)$. The right-hand tail of $\varphi_L$ is treated by the
same decomposition used to start the slide, using \cref{lm:phi_infty}.
Consequently there is $\eta\in(0,\bar s/2)$ for which the ordering holds for
every $\sigma\in[\bar s-\eta,\bar s]$; for $\sigma\ge\bar s$ it already holds
because $\bar s\in S$. Thus $\bar s-\eta\in S$, contradicting the definition
of $\bar s$. Therefore $\bar s=0$.

	Thus, for every $s>0$,
	$\psi_L(x+s)\le\psi_L(x)$ and
	$\varphi_L(x+s)\ge\varphi_L(x)$. Applying the strong comparison principle
	once more for this fixed $s$, using
	$\psi_L(L-s)>0=\psi^s(L-s)$ and the strict scalar monotonicity of
	$\varphi_L$ on $(L,+\infty)$, makes both inequalities strict wherever the
	two profiles are simultaneously defined. Hence $\psi_L$ is strictly
	decreasing and $\varphi_L$ strictly increasing. In particular
	$\psi_L'\le0$ and $\varphi_L'\ge0$; no differentiation of the merely
	Lipschitz function $A$ is required.
\end{proof}

\subsection{Endpoint behaviour of the boundary slope}

\begin{proof}[Proof of \eqref{eq:sw-largeL}]
	Put $w_L(x)=\psi_L(x+L)$ for $x\le0$. Then $w_L(0)=0$,
	$0<w_L\le u^*$, and
	\[
	-cw_L'-d_1w_L''
	\le(A(x+L)-c_1v^*)w_L .
	\]
	For $x\in(l_0-L,0)$ one has $A(x+L)=a_0$. With
	$\kappa=c_1v^*-a_0>0$,
	\[
	d_1w_L''+cw_L'-\kappa w_L\ge0 .
	\]
	If
	\[
	\beta=\frac{-c-\sqrt{c^2+4d_1\kappa}}{2d_1}<0,\qquad
	W_L(x)=u^*e^{\beta(x-l_0+L)},
	\]
	then $d_1W_L''+cW_L'-\kappa W_L=0$,
	$W_L(l_0-L)=u^*\ge w_L(l_0-L)$ and $W_L(0)>w_L(0)$.
	The maximum principle gives $w_L\le W_L$. Thus, for $L>l_0+2$,
	\[
	\|w_L\|_{L^\infty([-2,0])}
	\le u^*e^{\beta(L-l_0-2)}\longrightarrow0.
	\]
	The equation for $w_L$ is linear in $w_L$ with uniformly bounded
	coefficients. Standard one-dimensional boundary estimates on $[-2,0]$
	therefore give
	$\|w_L\|_{C^1([-1,0])}\to0$, and in particular
	$\psi_L'(L)=w_L'(0)\to0$.
\end{proof}

\begin{proof}[Proof of \eqref{eq:sw-smallL}]
	Let $L_n\downarrow0$ and choose semi-waves from the construction above. Set
	\[
	w_n(x)=\psi_{L_n}(x+L_n)\quad(x\le0),\qquad
	z_n(x)=\varphi_{L_n}(x+L_n)\quad(x\in\mathbb R).
	\]
	The bounds \eqref{s0}, monotonicity, and local ODE estimates give, after
	passing to a subsequence,
	\[
	w_n\to w\quad\text{in }C^2_{\rm loc}((-\infty,0])
	\quad\text{and}\quad
	z_n\to z\quad\text{in }C^2_{\rm loc}(\mathbb R),
	\]
	with the first convergence understood up to $x=0$ in $C^1$.
	For each compact subset of $(-\infty,0)$ one has
	$A(x+L_n)=a_1$ for all sufficiently large $n$, so $(w,z)$ solves the
	homogeneous semi-wave problem \eqref{semiwave0} at speed $c$ and $w(0)=0$.

	The barriers in \eqref{s0} imply
	$w(-\infty)=u^*$ and $z(-\infty)=v^*$. On $(0,+\infty)$ the limiting
	function $z$ solves
	$-cz'-d_2z''=(a_2-c_2z)z$, is nondecreasing and is bounded below by
	$v^*>0$; hence its limit at $+\infty$ is $a_2/c_2$. The strong maximum
	principle gives the strict monotonicity required in
	\cref{th:semi-wave0}. By uniqueness in that theorem,
	$(w,z)=(\Psi_c,\Phi_c)$. Since every convergent subsequence has the same
	limit, the whole sequence converges, and the $C^1$ convergence at the
	boundary yields
	\[
	\psi_{L_n}'(L_n)=w_n'(0)\longrightarrow\Psi_c'(0).
	\]
	As the sequence $L_n\downarrow0$ was arbitrary, \eqref{eq:sw-smallL}
	follows.
\end{proof}

\begin{remark}\label{rm:perturbed}
In \cref{sec_6} we use two slightly different perturbations.  For the upper
barrier we replace $(A,a_2)$ by $(A+2\delta,a_2-\delta)$; hence the favourable
plateau $a_1$ is replaced by $a_1+2\delta$.  For the lower barrier we replace
$(A,a_2)$ by $(A-2\delta,a_2+2\delta)$; hence the favourable plateau is
$a_1-2\delta$.  For $\delta>0$ sufficiently small all the structural
assumptions are preserved, in particular \eqref{weak_competition} and
$a_0+2\delta<0<a_1-2\delta$.  The extension of $\mu$ introduced after
\eqref{mu} makes the Stefan coefficient meaningful on the perturbed ranges.
The corresponding homogeneous critical speeds, denoted below by
$c_{0,\delta}^{+}$ and $c_{0,\delta}^{-}$, satisfy
\[
 c_{0,\delta}^{\pm}\longrightarrow c_0\qquad(\delta\downarrow0)
\]
by the parameter-continuity argument in \cref{rm:convert}.
\end{remark}

\section{The spreading speed: proof of \texorpdfstring{\cref{th:spreading_speed}}{the spreading speed theorem}}\label{sec_6}

Throughout this section we assume \eqref{weak_competition}, \eqref{initial} and
that spreading occurs, i.e. $h_\infty=+\infty$; by \cref{th1:lm2},
\begin{equation}\label{eq:spread_conv}
	\lim_{t\to+\infty}\big(u(t,x),v(t,x)\big)=(u^*,v^*)
	\qquad\text{ locally uniformly in }x\in[0,\infty).
\end{equation}
We distinguish the three cases $c<c_0$, $c=c_0$ and $c>c_0$. Several comparison
arguments below start at a positive time $T$; this is legitimate because
\cref{th7} holds with $0$ replaced by any initial time $T\ge0$, the data
$(u_0,v_0,h_0)$ being replaced by $(u(T,\cdot),v(T,\cdot),h(T))$. We always keep
the original time variable, so that the environment felt by the comparison
functions is exactly $A(\cdot-ct)$ and no relabelling of $A$ is needed.

\subsection{The case \texorpdfstring{$c<c_0$}{c<c0}}

\begin{lemma}\label{lem:limsup}
Suppose $c<c_0$. Then
\[
 \limsup_{t\to+\infty}\frac{h(t)}{t}\le c.
\]
\end{lemma}

\begin{proof}
For $\delta>0$ set
\[
 u^*_{\delta,+}:=
 \frac{(a_1+2\delta)c_2-(a_2-\delta)c_1}{b_1c_2-b_2c_1},
 \qquad
 v^*_{\delta,+}:=
 \frac{b_1(a_2-\delta)-(a_1+2\delta)b_2}{b_1c_2-b_2c_1}.
\]
These are the coexistence components for the upper perturbation
$(A,a_2)\mapsto(A+2\delta,a_2-\delta)$.  They satisfy
\begin{equation}\label{eq:order_delta}
 u^*<u^*_{\delta/2,+}<u^*_{\delta,+},\qquad
 v^*>v^*_{\delta/2,+}>v^*_{\delta,+}.
\end{equation}
By \cref{rm:perturbed}, choose $\delta>0$ so small that the perturbed problem
still satisfies all structural assumptions and $c<c_{0,\delta}^{+}$.
For each $L\ge0$, let $(\psi_{\delta,L},\varphi_{\delta,L})$ be a semi-wave
provided by \cref{th:semi-wave} for this perturbed problem at speed $c$.
By \eqref{eq:sw-largeL},
\[
 -\mu\big(A(L)+2\delta\big)\psi_{\delta,L}'(L)\longrightarrow0
 \qquad(L\to+\infty),
\]
because $\mu$ is bounded on the relevant compact range.  Hence we may fix
$L_\delta>0$ so large that, with
$(\psi_\delta,\varphi_\delta):=
(\psi_{\delta,L_\delta},\varphi_{\delta,L_\delta})$,
\begin{equation}\label{semi-wave-c}
 -\mu\big(A(L_\delta)+2\delta\big)\psi_\delta'(L_\delta)\le c.
\end{equation}
Thus
\begin{equation}\label{aux}
\begin{cases}
-c\psi_\delta'-d_1\psi_\delta''
 =(A(x)+2\delta-b_1\psi_\delta-c_1\varphi_\delta)\psi_\delta,
 &x<L_\delta,\\
-c\varphi_\delta'-d_2\varphi_\delta''
 =(a_2-\delta-b_2\psi_\delta-c_2\varphi_\delta)\varphi_\delta,
 &x\in\mathbb R,
\end{cases}
\end{equation}
with $\psi_\delta(-\infty)=u^*_{\delta,+}$,
$\psi_\delta=0$ on $[L_\delta,\infty)$,
$\psi_\delta'<0$, and
$\varphi_\delta(-\infty)=v^*_{\delta,+}$,
$\varphi_\delta(+\infty)=(a_2-\delta)/c_2$, $\varphi_\delta'>0$.

\textit{Step 1.} There exists $T_1>0$ such that
\begin{equation}\label{eq:step1}
 u(t,x)\le u^*_{\delta/2,+},\qquad
 v(t,x)\ge v^*_{\delta/2,+}
 \qquad(x\ge0,\ t\ge T_1).
\end{equation}
Indeed, this is \eqref{eq:uv_global} of \cref{lem:vlow} together with
\eqref{eq:order_delta}.

\textit{Step 2.} For every $T_*>0$ there are $T>T_*$ and $S_0>0$ such that
\begin{equation}\label{eq:step2}
 v(T,x)\ge\frac{a_2-\delta}{c_2}\qquad(x\ge S_0).
\end{equation}
Let $m_0=\inf_{x\ge0}v_0(x)>0$ and put
$m:=\frac12\min\{m_0,a_2/c_2\}$.  Consider the compatible half-line logistic
problem
\begin{equation}\label{aux1}
\begin{cases}
 w_t-d_2w_{xx}=(a_2-c_2w)w,&t>0,\ x>0,\\
 w(t,0)=0,&t>0,\\
 w(0,x)=m(1-e^{-x}),&x\ge0.
\end{cases}
\end{equation}
Standard scalar parabolic theory gives
$w(t,\cdot)\to w_*$ locally uniformly, where $w_*$ is the unique positive
stationary solution, $w_*'>0$ and $w_*(+\infty)=a_2/c_2$.  Moreover
$w(t,x+s)\ge w(t,x)$ for $s>0$, $x\ge0$: this follows by comparing the two
solutions on the translated half-line (their initial and boundary data are
ordered).  Thus $w(t,\cdot)$ is nondecreasing.  Choose $S_1>0$ so that
$w_*(S_1)>(a_2-\delta/2)/c_2$, and then choose $T>T_*$ so large that
$w(T,S_1)\ge(a_2-\delta)/c_2$.  Monotonicity gives
\begin{equation}\label{eq:w_lower}
 w(T,x)\ge\frac{a_2-\delta}{c_2}\qquad(x\ge S_1).
\end{equation}
Set $S_2:=h(T)$ and $\widetilde w(t,x):=w(t,x-S_2)$ for $x\ge S_2$.
For $0<t\le T$ and $x>S_2$ one has $x>h(t)$, hence $u(t,x)=0$ and
$v_t-d_2v_{xx}=(a_2-c_2v)v$.  Also
$\widetilde w(t,S_2)=0<v(t,S_2)$ and
$\widetilde w(0,x)\le m\le m_0\le v_0(x)$.  Scalar comparison therefore yields
$v(T,x)\ge w(T,x-S_2)$ for $x\ge S_2$.  Taking $S_0=S_1+S_2$ proves
\eqref{eq:step2}.

\textit{Step 3.} Apply Step 2 with $T_*=T_1$ and choose the resulting $T$ and
$S_0$.  Put $S:=\max\{S_0,h(T)\}$.  Since
$\psi_\delta(-\infty)=u^*_{\delta,+}>u^*_{\delta/2,+}$ and
$\varphi_\delta(-\infty)=v^*_{\delta,+}<v^*_{\delta/2,+}$, choose $R>0$ so
large that
\begin{equation}\label{eq:choiceR}
 cT+L_\delta+R>h(T),\qquad
 \psi_\delta(x-cT-R)>u^*_{\delta/2,+},\quad
 \varphi_\delta(x-cT-R)<v^*_{\delta/2,+}
 \quad(0\le x\le S).
\end{equation}
For $t\ge T$ define
\[
 \overline h(t)=ct+L_\delta+R,\qquad
 \overline u(t,x)=\psi_\delta(x-ct-R),\qquad
 \underline v(t,x)=\varphi_\delta(x-ct-R).
\]
Because $A$ is nonincreasing and $R>0$,
$A(x-ct-R)+2\delta\ge A(x-ct)$; hence \eqref{aux} gives the differential
inequalities required for an upper solution.  Also
$\overline u_x(t,0)\le0\le\underline v_x(t,0)$ and
$\overline u=0$ for $x\ge\overline h(t)$.
At the free boundary,
$\overline u_x(t,\overline h(t))=\psi_\delta'(L_\delta)$ and
$A(\overline h(t)-ct)=A(L_\delta+R)\le A(L_\delta)<A(L_\delta)+2\delta$.
Since $\mu$ is nondecreasing and $\psi_\delta'(L_\delta)<0$,
\eqref{semi-wave-c} yields
\[
 \overline h'(t)=c
 \ge-\mu\big(A(L_\delta)+2\delta\big)\psi_\delta'(L_\delta)
 \ge-\mu\big(A(\overline h(t)-ct)\big)
       \overline u_x(t,\overline h(t)).
\]
Finally, \eqref{eq:choiceR}, \eqref{eq:step1} and \eqref{eq:step2} give the
required ordering at time $T$: on $[0,h(T)]$ we have
$\overline u(T,\cdot)>u(T,\cdot)$; for $\underline v$ use
\eqref{eq:step1} on $[0,S]$ and \eqref{eq:step2} on $[S,\infty)$.
Therefore \cref{th7}(a) gives
$h(t)\le ct+L_\delta+R$ for all $t\ge T$, and the conclusion follows.
\end{proof}

\begin{lemma}\label{lem:liminf}
Suppose $c<c_0$. Then
\[
 \liminf_{t\to+\infty}\frac{h(t)}{t}\ge c.
\]
\end{lemma}

\begin{proof}
For $\delta>0$ set
\[
 u^*_{\delta,-}:=
 \frac{(a_1-2\delta)c_2-(a_2+2\delta)c_1}{b_1c_2-b_2c_1},
 \qquad
 v^*_{\delta,-}:=
 \frac{b_1(a_2+2\delta)-(a_1-2\delta)b_2}{b_1c_2-b_2c_1}.
\]
Thus $u^*_{\delta,-}<u^*$ and $v^*_{\delta,-}>v^*$.  By
\cref{rm:perturbed}, fix $\delta>0$ so small that the lower perturbed problem
$(A,a_2)\mapsto(A-2\delta,a_2+2\delta)$ satisfies the structural assumptions
and $c<c_{0,\delta}^{-}$.

Let $(\Psi^-_{s,\delta},\Phi^-_{s,\delta})$ denote its homogeneous semi-wave at
speed $s$, whose favourable invader growth rate is $a_1-2\delta$.  On
$[0,c_{0,\delta}^{-}]$ define
\[
 G_\delta(s):=-\mu(a_1-2\delta)
 (\Psi^-_{s,\delta})'(0)-s.
\]
By \cref{th:semi-wave0}, $G_\delta$ is continuous and has the unique zero
$c_{0,\delta}^{-}$.  The Hopf lemma gives $G_\delta(0)>0$; hence
$G_\delta(c)>0$ because $c<c_{0,\delta}^{-}$.
For each $L\ge0$, let $(\psi_{\delta,L},\varphi_{\delta,L})$ be a semi-wave
of the lower perturbed inhomogeneous problem at speed $c$.  By
\eqref{eq:sw-smallL} and $A(L)\to a_1$ as $L\downarrow0$,
\[
 -\mu\big(A(L)-2\delta\big)\psi_{\delta,L}'(L)
 \longrightarrow
 -\mu(a_1-2\delta)(\Psi^-_{c,\delta})'(0)>c.
\]
Consequently, for some sufficiently small $L_\delta>0$, with
$(\psi_\delta,\varphi_\delta):=
(\psi_{\delta,L_\delta},\varphi_{\delta,L_\delta})$,
\begin{equation}\label{eq:sw2}
 -\mu\big(A(L_\delta)-2\delta\big)\psi_\delta'(L_\delta)\ge c.
\end{equation}
The profiles satisfy
\begin{equation}\label{aux2}
\begin{cases}
-c\psi_\delta'-d_1\psi_\delta''
 =(A(x)-2\delta-b_1\psi_\delta-c_1\varphi_\delta)\psi_\delta,
 &x<L_\delta,\\
-c\varphi_\delta'-d_2\varphi_\delta''
 =(a_2+2\delta-b_2\psi_\delta-c_2\varphi_\delta)\varphi_\delta,
 &x\in\mathbb R,
\end{cases}
\end{equation}
with the usual strict monotonicity and limits
$\psi_\delta(-\infty)=u^*_{\delta,-}$,
$\psi_\delta=0$ on $[L_\delta,\infty)$,
$\varphi_\delta(-\infty)=v^*_{\delta,-}$ and
$\varphi_\delta(+\infty)=(a_2+2\delta)/c_2$.

Let $V$ solve $V'=(a_2-c_2V)V$ with
$V(0)=\|v_0\|_\infty$.  Comparison gives $v(t,x)\le V(t)$, so there is
$T_2>0$ such that
\begin{equation}\label{eq:v}
 v(t,x)<\frac{a_2+\delta}{c_2}
 \qquad(t\ge T_2,\ x\ge0).
\end{equation}
Choose $R_1\ge L_\delta$ with
\begin{equation}\label{eq:vp}
 \varphi_\delta(x)>\frac{a_2+\delta}{c_2}\qquad(x\ge R_1).
\end{equation}
By \eqref{eq:spread_conv}, $h_\infty=\infty$, and
$u^*_{\delta,-}<u^*$, $v^*_{\delta,-}>v^*$, there is $T_0>T_2$ such that
\begin{equation}\label{eq:uv}
 h(T_0)>L_\delta,\qquad
 u(t,x)>u^*_{\delta,-},\quad v(t,x)<v^*_{\delta,-}
 \qquad(t\ge T_0,\ 0\le x\le R_1).
\end{equation}

For $t\ge T_0$ define
\[
 \underline h(t)=c(t-T_0)+L_\delta,\qquad
 \underline u(t,x)=\psi_\delta(x-c(t-T_0)),\qquad
 \overline v(t,x)=\varphi_\delta(x-c(t-T_0)).
\]
Since $A$ is nonincreasing and $cT_0>0$,
$A(x-c(t-T_0))\le A(x-ct)$; therefore \eqref{aux2} yields the differential
inequalities required for a lower solution.  At its free boundary,
\[
 \underline h(t)-ct=L_\delta-cT_0<L_\delta,
\]
so $A(\underline h(t)-ct)\ge A(L_\delta)>A(L_\delta)-2\delta$.
Since $\mu$ is nondecreasing and $\psi_\delta'(L_\delta)<0$,
\eqref{eq:sw2} gives
\[
 \underline h'(t)=c
 \le-\mu\big(A(L_\delta)-2\delta\big)\psi_\delta'(L_\delta)
 \le-\mu\big(A(\underline h(t)-ct)\big)
       \underline u_x(t,\underline h(t)).
\]
At $x=0$ we use the second alternative in \cref{th7}(b): for $t\ge T_0$ the
monotonicity of the profiles gives
\[
 \underline u(t,0)<u^*_{\delta,-},\qquad
 \overline v(t,0)>v^*_{\delta,-}.
\]
The corresponding inequalities with $u(t,0),v(t,0)$ hold for every
$t\ge T_0$ by \eqref{eq:uv}.
At $t=T_0$, \eqref{eq:uv} gives the ordering of $h$ and $u$; for $v$, use
\eqref{eq:uv} on $[0,R_1]$ and \eqref{eq:v}--\eqref{eq:vp} on $(R_1,\infty)$.
Thus \cref{th7}(b) implies
$h(t)\ge c(t-T_0)+L_\delta$ for $t\ge T_0$, and the conclusion follows.
\end{proof}

\subsection{The case \texorpdfstring{$c\ge c_0$}{c>=c0}}

\begin{lemma}\label{lm:limsup_c0}
	Suppose $c\ge c_0$. Then $\displaystyle\limsup_{t\to+\infty}\frac{h(t)}{t}\le c_0$.
\end{lemma}

\begin{proof}
	As in Step~0 of the proof of \cref{th1:lm2}, $(u,v,h)$ is a lower solution of
	\eqref{main2}; let $(\overline u,\underline v,\overline h)$ be the solution of
	\eqref{main2} with the same initial data, so that $h\le\overline h$. Since
	$h_\infty=+\infty$ we get $\overline h(t)\to+\infty$, i.e. spreading occurs
	for \eqref{main2}, and \cite[Theorem 1.3]{MR3986328} gives
	$\overline h(t)/t\to c_0$. Hence $\limsup_{t\to+\infty}h(t)/t\le c_0$.
\end{proof}

\begin{lemma}\label{lm:eq}
	Suppose $c = c_0$. Then $\displaystyle\lim_{t\to+\infty} \frac{h(t)}{t} = c_0$.
\end{lemma}

\begin{proof}
	By \cref{lm:limsup_c0} it suffices to prove
	$\liminf_{t\to+\infty}h(t)/t\ge c_0$. For small $\delta>0$ put
	$c_\delta:=c_0-\delta>0$ and let $(u_\delta,v_\delta,h_\delta)$ be the
	solution of \eqref{main} with $c$ replaced by $c_\delta$ and with the same
	initial data. Since $c_\delta<c$ and $A$ is nonincreasing, we have
	$A(x-ct)\ge A(x-c_\delta t)$ and $\mu(A(h(t)-ct))\ge\mu(A(h(t)-c_\delta t))$
	for $t>0$; as $-u_x(t,h(t))>0$, this shows that $(u,v,h)$ is an upper solution
	of the $c_\delta$-problem in the sense of \cref{th7}(a) (first alternative at
	$x=0$). Hence
	\begin{equation}\label{eq:cdelta}
		h_\delta(t)\le h(t)\qquad (t\ge0).
	\end{equation}

	We next check that spreading occurs for the $c_\delta$-problem when $\delta$
	is small. Since $h_\infty=+\infty$, there is $T>0$ with $h(T)>R^*+1$. The
	solution of \eqref{main} depends continuously on the parameter $c$, uniformly
	on the fixed interval $[0,T]$ (this follows from the contraction mapping
	construction of \cref{th4}, whose constants depend continuously on $c$,
	together with uniqueness and a continuation argument on $[0,T]$, the solution
	being global with bounds uniform in $c$ by \cref{th6}), so we may fix
	$\delta>0$ small enough that $h_\delta(T)>R^*$. If vanishing occurred for the
	$c_\delta$-problem, \cref{th1:lm1} would give $h_{\delta,\infty}\le R^*$,
	contradicting $h_\delta(T)>R^*$. Hence spreading occurs for the
	$c_\delta$-problem.

	Since $c_\delta<c_0$, \cref{lem:liminf} applied to the $c_\delta$-problem
	yields $\liminf_{t\to+\infty}h_\delta(t)/t\ge c_\delta$, and by
	\eqref{eq:cdelta},
	\[
	\liminf_{t\to+\infty}\frac{h(t)}{t}\ \ge\ \liminf_{t\to+\infty}\frac{h_\delta(t)}{t}\ \ge\ c_0-\delta .
	\]
	Letting $\delta\to0$ completes the proof.
\end{proof}

\begin{lemma}\label{lm:g}
	Suppose $c > c_0$. Then $\displaystyle\lim_{t\to+\infty} \frac{h(t)}{t} = c_0$.
\end{lemma}

\begin{proof}
	By \cref{lm:limsup_c0}, $\limsup_{t\to+\infty}h(t)/t\le c_0<c$, so there
	exists $T_0>0$ with $h(t)<ct$ for all $t>T_0$. By \cref{lem:reduction},
	$(t,x)\mapsto(u(t+T_0,x),v(t+T_0,x),h(t+T_0))$ solves \eqref{main2} with data
	$(u(T_0,\cdot),v(T_0,\cdot),h(T_0))$, which satisfy \eqref{initial} (the
	condition $\liminf_{x\to\infty}v(T_0,x)>0$ following from \cref{lem:vlow}).
	Since $h_\infty=+\infty$, spreading occurs for that problem, and
	\cite[Theorem 1.3]{MR3986328} gives $h(t)/t\to c_0$.
\end{proof}

\begin{proof}[Proof of \cref{th:spreading_speed}]
	Combine \cref{lem:limsup}, \cref{lem:liminf}, \cref{lm:eq} and \cref{lm:g}.
\end{proof}

\section{Numerical simulations}\label{sec_7}

In this section we illustrate the theoretical results by numerical experiments.
Problem \eqref{main} is discretised by a finite difference scheme on the
straightened domain, combined with a Newton--Raphson iteration at each time step.
The time-dependent computations are used only as qualitative illustrations of the analytical results and are not used in any proof.

In all the figures below the parameters are fixed as
\[
d_1 = 0.02,\quad a_1 = 0.14,\quad b_1 = 0.4,\quad c_1 = 0.11,\quad
d_2 = 0.018,\quad a_2 = 0.18,\quad b_2 = 0.1,\quad c_2 = 0.2,
\]
\[
a_0 = -0.2,\quad l_0 = 1,\quad \mu(s) = 0.2 + 0.1s,\quad
A(\xi) = \begin{cases}
	a_1, & \xi \le 0, \\
	a_0 + \frac{1}{2} (a_1 - a_0)\big(\cos( \pi \xi / l_0 ) + 1\big), & 0 < \xi < l_0, \\
	a_0, & \xi \ge l_0,
\end{cases}
\]
and the initial data are
$u_0(x) = 0.2 \cos\big(\frac{\pi x}{2h_0}\big)$ and
$v_0(x) = 0.2 \cos\big(\frac{\pi x}{2}\big) + 0.3$, where $h_0$ is the initial
size of the invaded range. These values satisfy all our standing assumptions:
$\mu$ is increasing on $[a_0,a_1]$ with $\mu(a_0)=0.18>0$, and
\[
\frac{c_1}{c_2} = 0.55 < \frac{a_1}{a_2}\approx0.778 < \frac{b_1}{b_2} = 4,
\]
so that \eqref{weak_competition} holds. The corresponding steady states are
\[
u^*\approx0.1188,\qquad v^*\approx0.8406,\qquad \frac{a_2}{c_2}=0.9,
\]
and the two thresholds of \eqref{Rstar} are $R_0\approx0.594$ and
$R^*\approx1.097$.

\subsection{The spreading--vanishing dichotomy}

We first fix the shifting speed at $c=0.04$ and vary $h_0$.

\smallskip
\noindent\textbf{Vanishing of $u$.} For $h_0=0.3$ (note $h_0<R_0$) the invasive
species fails to establish itself and dies out; the free boundary converges to
$h_\infty\approx0.91$, which is consistent with the theoretical bound
$h_\infty\le R^*\approx1.097$ of \cref{th:dynamics}(ii). Relieved from the
competitive pressure of $u$, the native species recovers its own carrying
capacity $a_2/c_2=0.9$. See \cref{fig:u_vanishing,fig:u_vanishing_3d,fig:u_vanishing_ht}.

\begin{figure}[htbp]
	\centering
	\begin{subfigure}{0.45\textwidth}
		\centering
		\includegraphics[width=\linewidth]{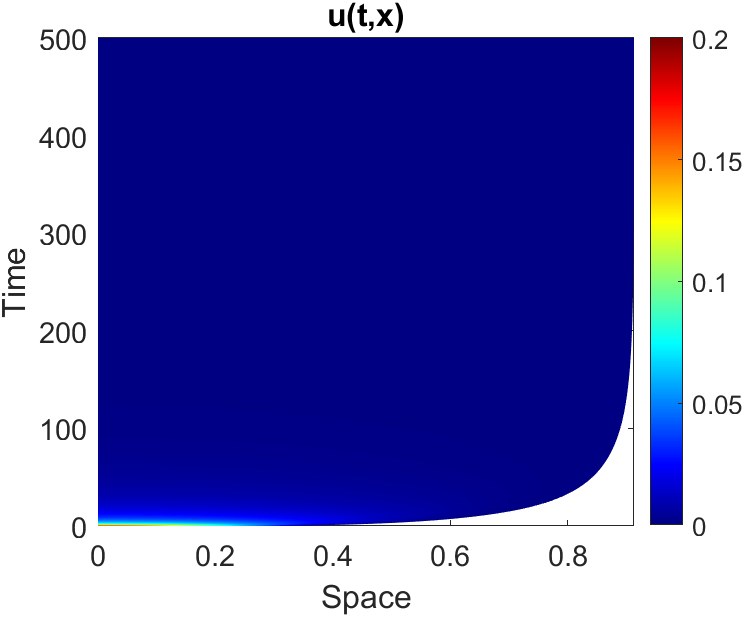}
		\caption{}
		\label{vanishing_u}
	\end{subfigure}
	\hfill
	\begin{subfigure}{0.45\textwidth}
		\centering
		\includegraphics[width=\linewidth]{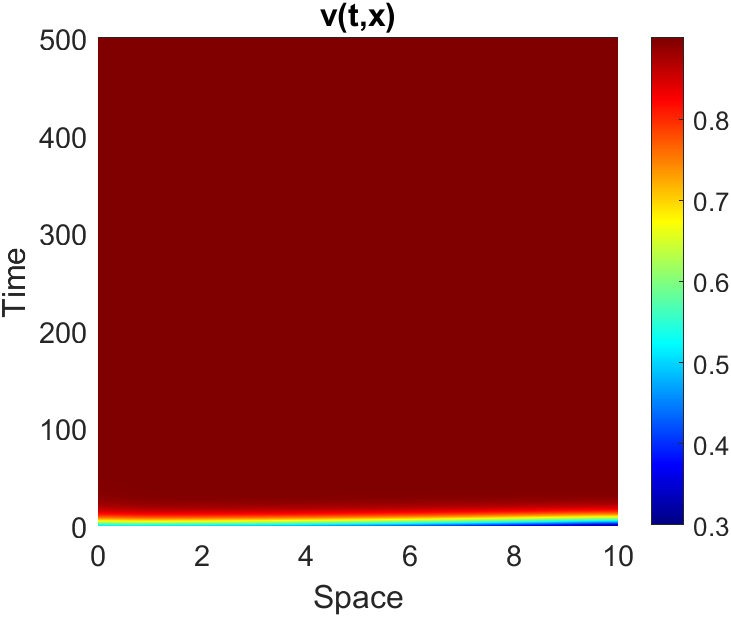}
		\caption{}
		\label{vanishing_v}
	\end{subfigure}
	\caption{Vanishing of the invasive species with $c=0.04$ and $h_0 = 0.3$: (A) heatmap of $u(t,x)$; (B) heatmap of $v(t,x)$.}
	\label{fig:u_vanishing}
\end{figure}

\begin{figure}[htbp]
	\centering
	\begin{subfigure}{0.45\textwidth}
		\centering
		\includegraphics[width=\linewidth]{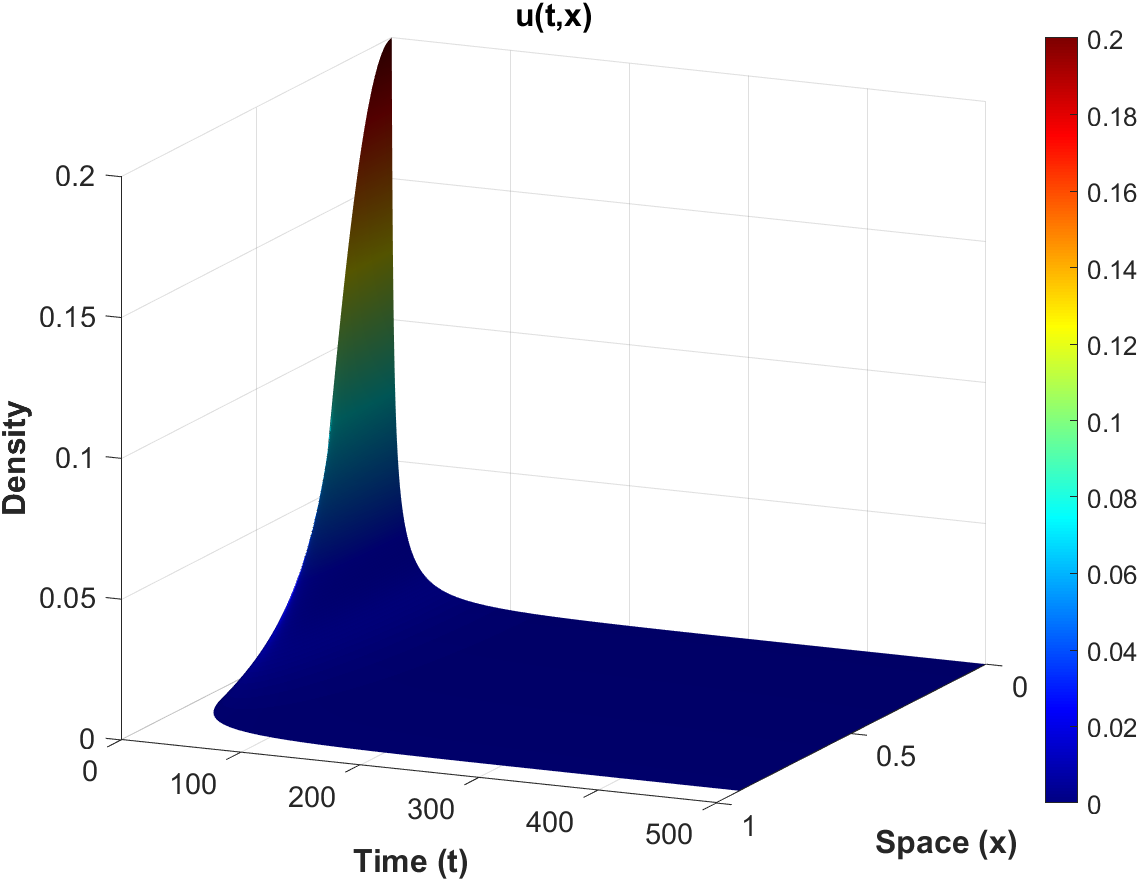}
		\caption{}
		\label{vanishing_u_3d}
	\end{subfigure}
	\hfill
	\begin{subfigure}{0.45\textwidth}
		\centering
		\includegraphics[width=\linewidth]{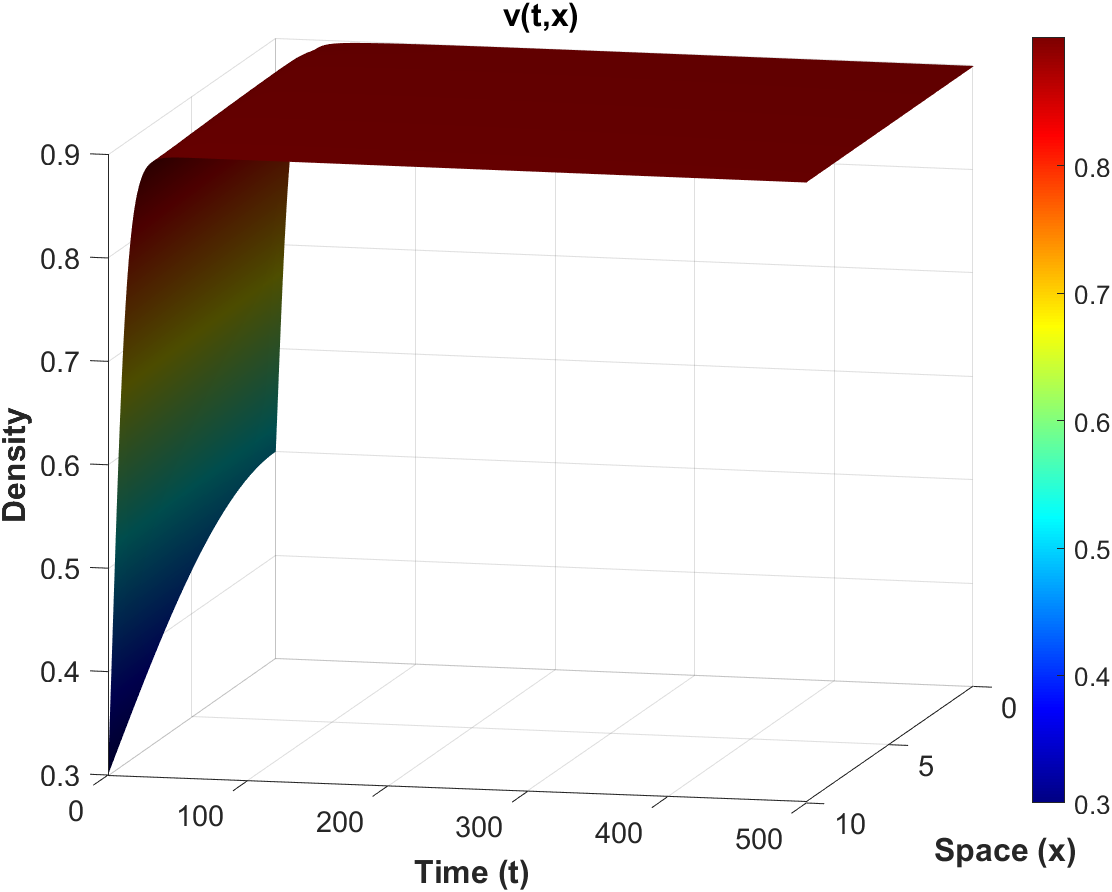}
		\caption{}
		\label{vanishing_v_3d}
	\end{subfigure}
	\caption{The same simulation as in \cref{fig:u_vanishing}, shown as surface plots: (A) $u(t,x)$; (B) $v(t,x)$.}
	\label{fig:u_vanishing_3d}
\end{figure}

\begin{figure}[htbp]
	\centering
	\includegraphics[scale=0.8]{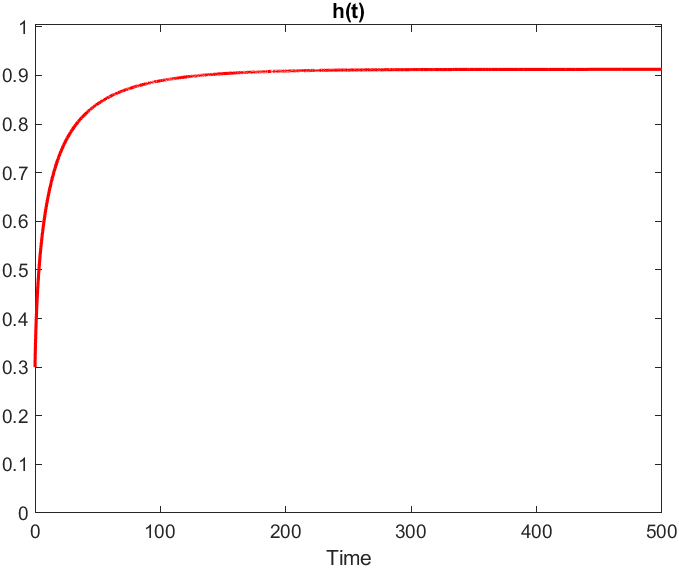}
	\caption{The free boundary $h(t)$ when $u$ vanishes ($c = 0.04$, $h_0 = 0.3$). Here $h_\infty\approx0.91\le R^*\approx1.097$.}
	\label{fig:u_vanishing_ht}
\end{figure}

\noindent\textbf{Spreading of $u$.} For $h_0=1$ the invasion succeeds: the range
$[0,h(t)]$ expands indefinitely and the two densities converge to the coexistence
state $(u^*,v^*)\approx(0.12,0.84)$. Ahead of the invasion front the native
species sits at its own carrying capacity $a_2/c_2=0.9$; behind the front it
drops to $v^*\approx0.84$, so that the invasion depresses, but does not
eliminate, the resident population. See \cref{fig:heat_spread,fig:byw}. Note that
$h_0=1<R^*\approx1.097$, so this run also shows that the sufficient condition
$h_0\ge R^*$ of \cref{th:criteria}(i) is not necessary.

\begin{figure}[htbp]
	\centering
	\begin{subfigure}{0.45\textwidth}
		\centering
		\includegraphics[width=\linewidth]{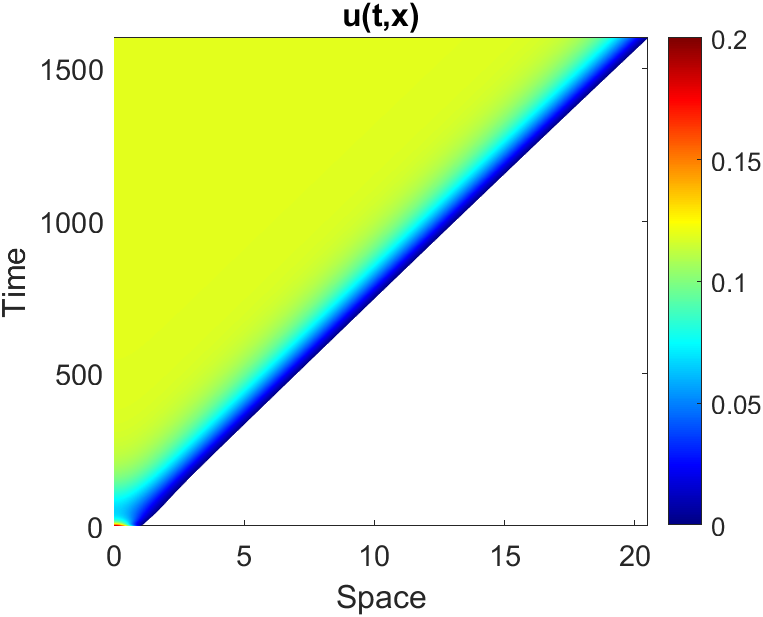}
		\caption{}
		\label{spreading_heat_u}
	\end{subfigure}
	\hfill
	\begin{subfigure}{0.45\textwidth}
		\centering
		\includegraphics[width=\linewidth]{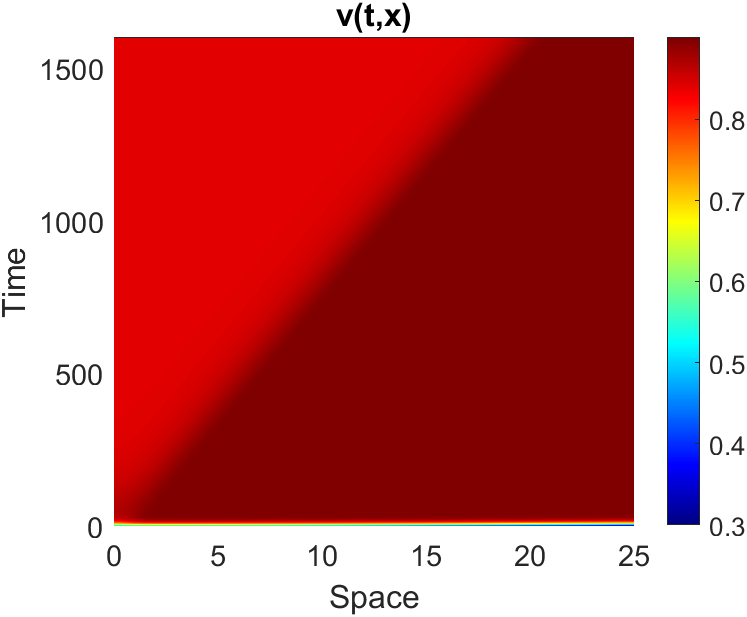}
		\caption{}
		\label{spreading_heat_v}
	\end{subfigure}
	\caption{Spreading of the invasive species with $c=0.04$ and $h_0=1$: (A) heatmap of $u(t,x)$; (B) heatmap of $v(t,x)$. The straight interface is the invasion front $x=h(t)$.}
	\label{fig:heat_spread}
\end{figure}

\begin{figure}[htbp]
	\centering
	\begin{subfigure}{0.45\textwidth}
		\centering
		\includegraphics[width=\linewidth]{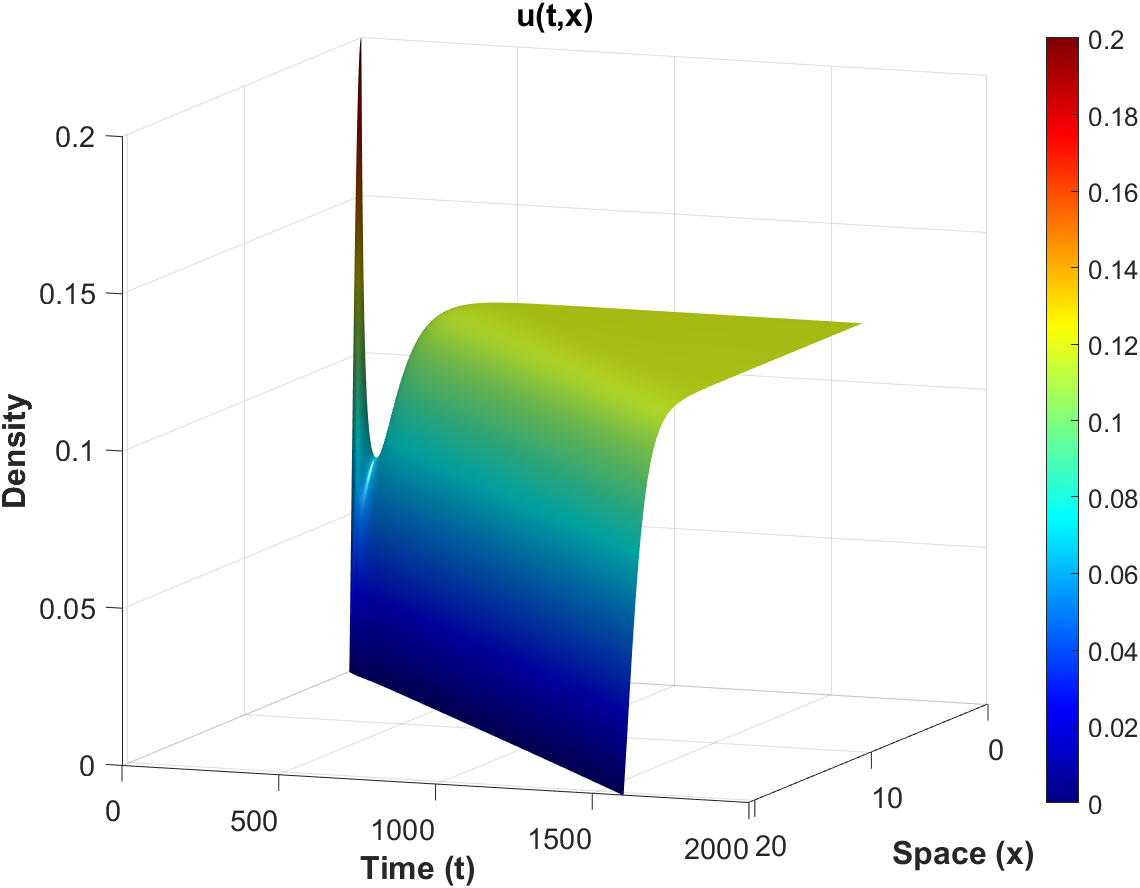}
		\caption{}
		\label{u_spreading_a}
	\end{subfigure}
	\hfill
	\begin{subfigure}{0.45\textwidth}
		\centering
		\includegraphics[width=\linewidth]{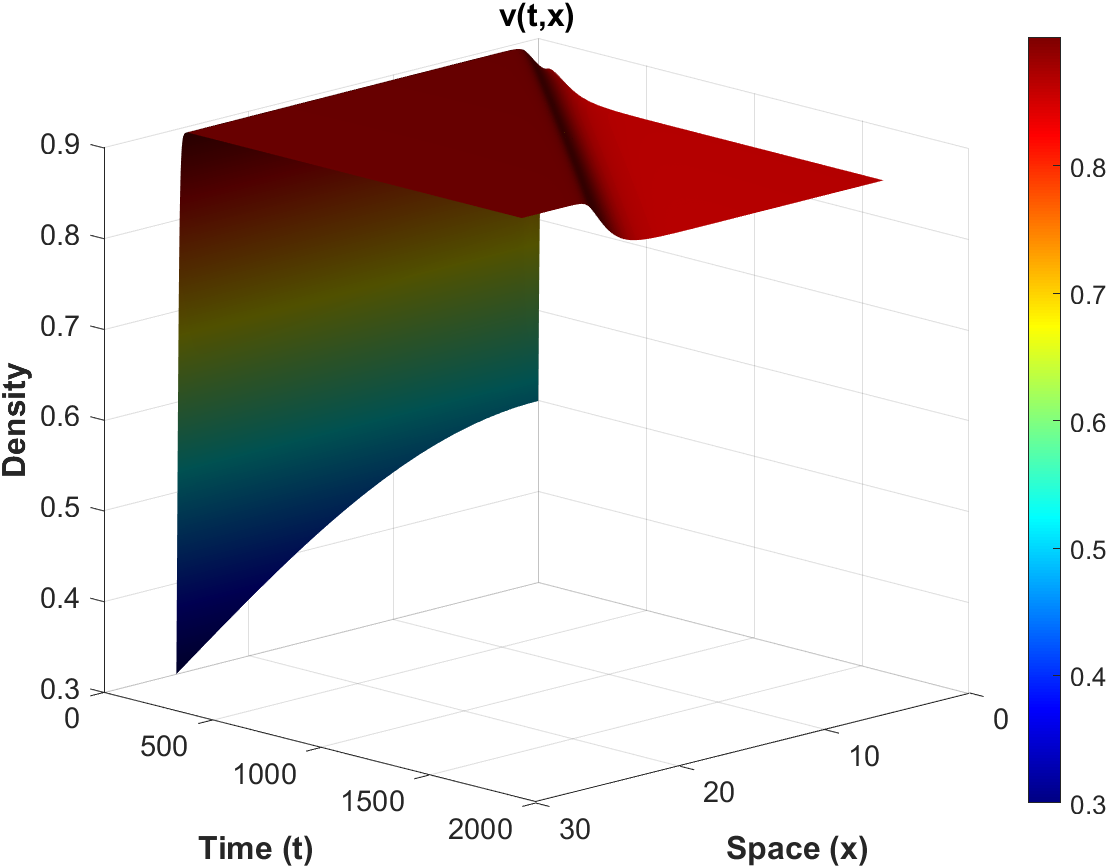}
		\caption{}
		\label{u_spreading_b}
	\end{subfigure}
	\caption{The same simulation as in \cref{fig:heat_spread}, shown as surface plots: (A) $u(t,x)$ on the moving domain; (B) $v(t,x)$ on the fixed domain.}
	\label{fig:byw}
\end{figure}

\subsection{The asymptotic spreading speed}

We now fix $h_0=1$, so that spreading occurs, and vary the shifting speed $c$.
To distinguish the asymptotic selected speed from a finite-time value of
$h(t)/t$, we also solve the homogeneous semi-wave boundary-value problem
\eqref{semiwave0} on successively larger truncated intervals and enforce the
Stefan condition \eqref{c0}. The computed value is
\[
 c_0\approx0.01208.
\]
The displayed digits are unchanged when the truncation interval is enlarged.
For comparison, the analogous scalar semi-wave calculation for
\eqref{single_hom} gives $c_0^{\rm sing}\approx0.03872$, about $3.2$ times
larger, in agreement with \cref{lm:c0sing}.

\Cref{fig:hovertc004} displays the ratio $h(t)/t$ for $c=0.04$. At the final
simulation time $t\approx1600$ it is about $0.0128$, strictly below $c$ but still
slightly above the independently computed limit $c_0\approx0.01208$.
\Cref{fig:compare} and \cref{table1} show the same finite-time effect for
several values of $c$: for $c\le0.01$ the ratio is close to $c$, whereas for
$c\ge0.02$ the curves cluster near $0.0125$--$0.0128$ and decrease slowly toward
$c_0$. Thus the simulations are consistent with
$\lim_{t\to\infty}h(t)/t=\min\{c,c_0\}$. The small excess of the tabulated
values over the limiting speed is a finite-time transient; in particular, the
figures should not be read as proving an $O(1)$ phase shift or as identifying
$c_0$ merely from the value of $h(t)/t$ at $t=1600$.

\begin{figure}[htbp]
	\centering
	\begin{subfigure}{0.45\textwidth}
		\centering
		\includegraphics[width=\linewidth]{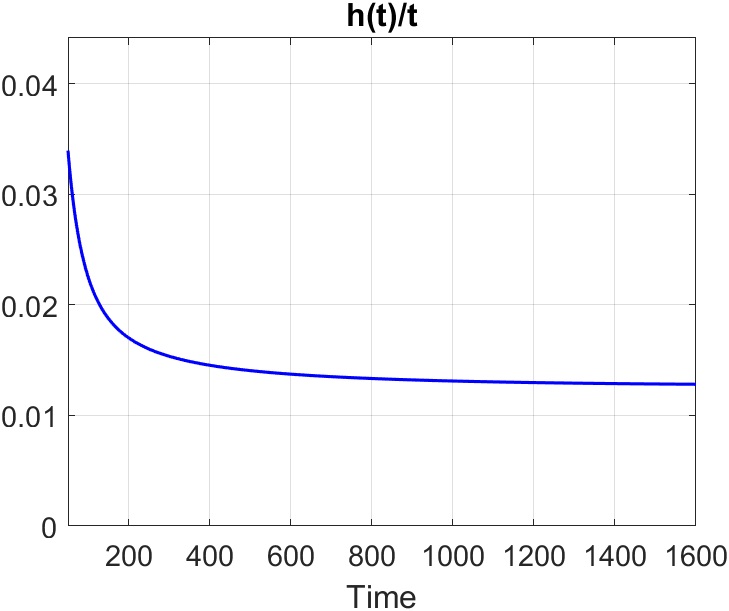}
		\caption{}
		\label{hovertc004_a}
	\end{subfigure}
	\hfill
	\begin{subfigure}{0.45\textwidth}
		\centering
		\includegraphics[width=\linewidth]{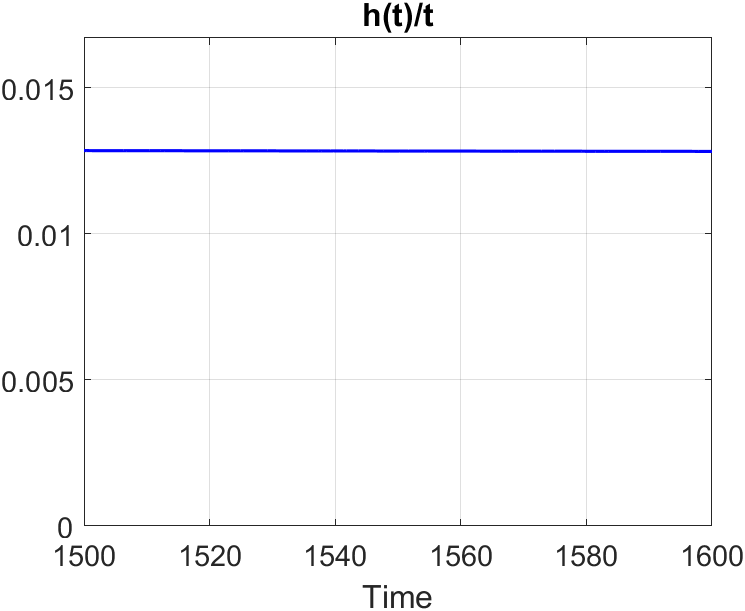}
		\caption{}
		\label{hovertc004_b}
	\end{subfigure}
	\caption{The ratio $h(t)/t$ for $c = 0.04$: (A) over $t\in[50, 1600]$; (B) a zoom for large $t$.}
	\label{fig:hovertc004}
\end{figure}

\begin{table}[htbp]
	\centering
	\begin{tabular}{c|cccccccc}
		$c$ & $0.007$ & $0.008$ & $0.009$ & $0.01$ & $0.02$ & $0.03$ & $0.04$ & $0.05$\\
		\hline
		$h(t)/t$ at $t\approx1600$ & $0.0074$ & $0.0084$ & $0.0094$ & $0.0103$ & $0.0125$ & $0.0127$ & $0.0128$ & $0.0128$
	\end{tabular}
	\caption{The ratio $h(t)/t$ at $t\approx1600$ for selected values of $c$ ($h_0=1$). The values are finite-time approximations consistent with $\min\{c,c_0\}$; the semi-wave computation gives $c_0\approx0.01208$.}
	\label{table1}
\end{table}

\begin{figure}[htbp]
	\centering
	\includegraphics[scale=0.8]{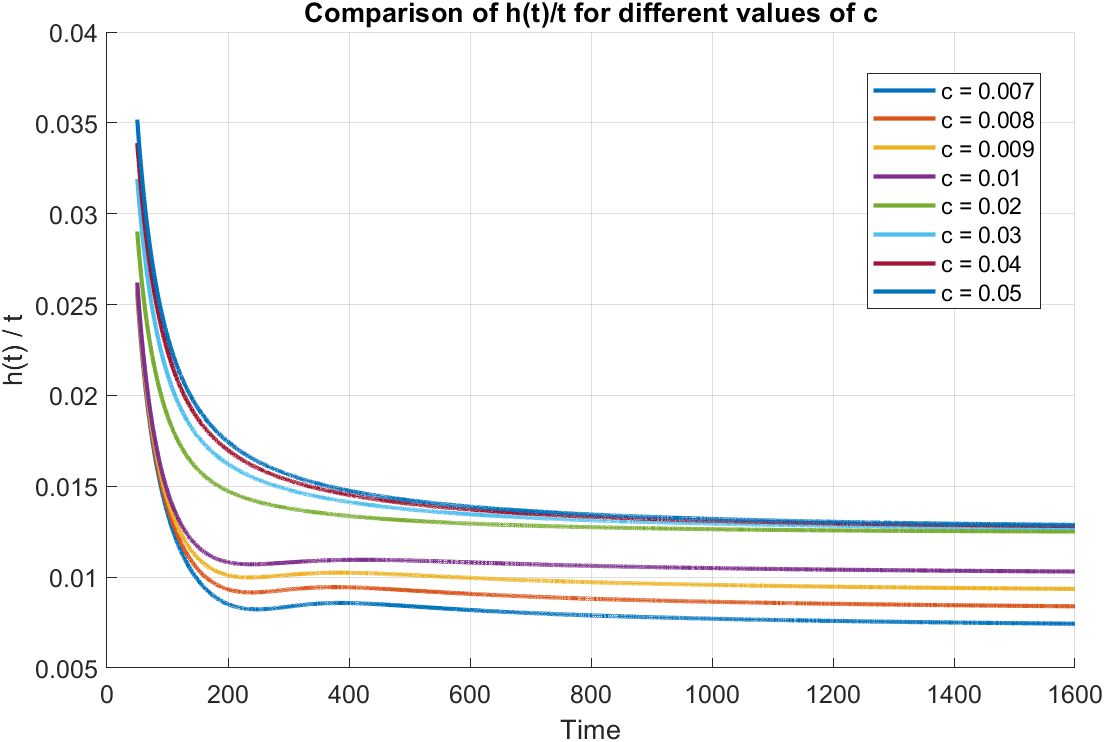}
	\caption{Comparison of $h(t)/t$ over $t \in [50,1600]$ for selected values of $c$.}
	\label{fig:compare}
\end{figure}

\section{Concluding remarks}\label{sec_8}

We have determined the long-time dynamics and the exact asymptotic spreading
speed of a free boundary competition system in which an invasive species benefits
from a climate shift while the native competitor occupies a fixed, globally
favourable habitat, in the weak competition regime where the two species may
coexist. The dichotomy of \cref{th:dynamics}, the criteria of
\cref{th:criteria} and the formula $\min\{c,c_0\}$ of
\cref{th:spreading_speed} complete, for the model \eqref{main}, the picture
started in \cite{DLN2026a} for the weak--strong regimes.

Several natural questions remain open. First, as in \cite{HU20205931}, we do not
know whether the threshold $\sigma_0$ of \cref{th:criteria}(ii) can be infinite.
Second, when $c<c_0$ our result gives the speed of the front but not its precise
profile; in the single species case \cite[Theorem 1.2]{HU20205931} shows that
$h(t)-ct\to L_0$ and that the solution converges to a selected inhomogeneous
semi-wave. It is natural to expect an analogous statement here, with a suitably
selected pair from the family constructed in \cref{th:semi-wave}; establishing
the required selection and uniqueness is a separate problem. Third, the strong competition case is not covered by the present argument.
In that regime the kinetic competition system is bistable rather than globally
attracting, so the monotone-iteration input and the monostable semi-waves used
here are no longer the appropriate objects. Developing a spreading criterion
and a speed-selection mechanism for the corresponding shifting free-boundary
problem therefore appears to require a different bistable framework. Fourth,
one may ask what happens when the native competitor is itself subject to a
shifting environment. If its growth rate is
replaced by a function $\tilde a_2(x-ct)$ moving at the \emph{same} speed $c$,
the common moving frame $\xi=x-ct$ remains available and an analogue of the
semi-wave formulation in \cref{sec_5} is still natural. However, both the
homogeneous input from \cite{MR3986328} and the comparison barriers in
\cref{sec_6} would have to be rebuilt, because the resident end states are then
spatially heterogeneous rather than constant. The case of two climatic edges
moving at different speeds admits no common autonomous moving frame and appears
to require substantially different techniques. Finally, it
would be interesting to treat the corresponding model with nonlocal diffusion, in
the spirit of \cite{MR4873231}.

\appendix

\section{Proof of \texorpdfstring{\cref{th4}}{the local existence theorem}}\label{sec_app}

We follow the strategy of \cite[Theorem 2.1]{MR3327894}; the details are included
for completeness, since the shifting terms $A(\cdot-ct)$ and $\mu(A(\cdot))$
require some care.

\textit{Step 1: straightening the free boundary.} Let $\zeta\in C^3([0,\infty))$
satisfy
\[
\zeta(y) = 1 \text{ for } |y - h_0| < \tfrac{h_0}{8}, \quad \zeta(y) = 0 \text{ for } |y - h_0| > \tfrac{h_0}{2}, \quad |\zeta'(y)| < \tfrac{6}{h_0}\ \ \forall y,
\]
and set $x=F(t,y):=y+\zeta(y)(h(t)-h_0)$ for $y\ge0$. If
$|h(t)-h_0|\le\frac{h_0}{8}$, then
$\partial_yF=1+\zeta'(y)(h(t)-h_0)\ge1-\frac{6}{h_0}\cdot\frac{h_0}{8}>0$, so
$F(t,\cdot)$ is a diffeomorphism of $[0,\infty)$ mapping $[0,h_0]$ onto
$[0,h(t)]$. Writing
\[
\sqrt{B(h,y)}:=\frac{1}{1+\zeta'(y)(h-h_0)},\quad
C(h,y):=-\frac{\zeta''(y)(h - h_0)}{[1+\zeta'(y)(h-h_0)]^3},\quad
D(h,y):=\frac{\zeta(y)}{1 + \zeta'(y)(h - h_0)},
\]
and $w(t,y):=u(t,F(t,y))$, $z(t,y):=v(t,F(t,y))$, problem \eqref{main} becomes
\begin{equation}\label{eq:transform}
	\begin{cases}
		w_t - d_1 B w_{yy} - (d_1C + h' D) w_y = \big(A(F(t,y) - ct) - b_1w - c_1z \big)w, & t > 0, ~ 0 \le y < h_0,\\
		z_t - d_2 B z_{yy} - (d_2C + h' D) z_y = \left(a_2 -b_2w -c_2z \right)z, & t >0, ~ 0 \le y < \infty,\\
		w_y(t,0) = z_y(t,0) = 0, \quad w(t, y) = 0 \ (y\ge h_0), & t > 0,\\
		h'(t) = -\mu\big(A(h(t) - ct) \big)  w_y(t, h_0), & t > 0,\\
		h(0) = h_0, ~ w(0, y) = u_0(y),\ z(0, y) = v_0(y). &
	\end{cases}
\end{equation}
Extend $u_0$ by zero to $[h_0,\infty)$ and denote the extension by
$\widetilde u_0$. Set $h^* := -\mu(A(h_0))u_0'(h_0)$,
$\Delta_T := [0,T] \times [0,h_0]$ and
$\Delta_T^\infty := [0,T] \times [0,\infty)$. For
$0 < T < \frac{h_0}{8(1+h^*)}$ define $\Gamma_T := W_T \times Z_T \times H_T$,
where
\begin{align*}
	H_T &= \left\{h \in C^1([0, T]) : h(0) = h_0, ~ h'(0) = h^*, ~ \|h' - h^*\|_{C([0,T])} \le 1 \right\}, \\
	W_T &= \left\{ w \in C(\Delta_T^\infty) : w \equiv 0 \text{ for } y \ge h_0, ~ w(0,\cdot) = \widetilde u_0, ~ \|w - u_0\|_{C(\Delta_T)} \le 1 \right\}, \\
	Z_T &= \left\{ z \in C_b(\Delta_T^\infty) : z(0,\cdot) = v_0, ~ \|z - v_0\|_{L^\infty(\Delta_T^\infty)} \le 1 \right\},
\end{align*}
a complete metric space for
$\mathcal{D}\big((w_1, z_1, h_1), (w_2, z_2, h_2)\big) = \|w_1 - w_2\|_{C(\Delta_T)}
+ \|z_1 - z_2\|_{L^\infty(\Delta_T^\infty)} + \|h_1' - h_2'\|_{C([0,T])}$.

\textit{Step 2: a priori estimates.} For $h_1,h_2 \in H_T$, since
$h_1(0)=h_2(0)$,
\begin{equation}\label{eq:h12}
	\|h_1 - h_2\|_{C([0,T])} \le T \|h_1' - h_2'\|_{C([0,T])}.
\end{equation}
Given $(w,z,h)\in\Gamma_T$ put
$f(t,y):=\big(A(F(t,y)-ct)-b_1w-c_1z\big)w$ and
$g(t,y):=(a_2-b_2w-c_2z)z$. Since $\|w\|_\infty\le M_1':=\|u_0\|_{C([0,h_0])}+1$,
$\|z\|_\infty\le M_2':=\|v_0\|_{L^\infty}+1$ and $A\le a_1$, both $f$ and $g$ are
bounded by constants depending only on the data. By $L^p$ theory and Sobolev
embedding, the linear problem obtained from \eqref{eq:transform} by freezing
$(w,z,h)$ in the right hand side and in the coefficients admits a unique bounded
solution $(\tilde w,\tilde z)$ with
\begin{equation}\label{eq:wz}
	\|\tilde{w} \|_{C^{(1+\alpha)/2,1+\alpha}(\Delta_T)} \le C_1, \qquad \|\tilde{z}\|_{C^{(1+\alpha)/2,1+\alpha}(\Delta_T^\infty)} \le C_1,
\end{equation}
where $C_1$ depends on $\alpha,h_0,c,a_0,a_1,a_2,b_i,c_i,d_i$,
$\|u_0\|_{C^2([0,h_0])}$ and $\|v_0\|_{C_b^2([0,\infty))}$. Define
\[
\tilde{h}(t) := h_0 - \int_0^t \mu \big( A(h(\tau) - c\tau) \big) \tilde{w}_y(\tau, h_0)\,\mathrm{d}\tau ,
\]
so that $\tilde h(0)=h_0$, $\tilde h'(0)=h^*$ and, $\mu\circ A$ being Lipschitz,
\begin{equation}\label{eq:hprime}
	\|\tilde{h}'\|_{C^{\alpha/2}([0,T])} \le C_2 ,
\end{equation}
with $C_2$ depending on $C_1$, $c$ and the Lipschitz constants of $\mu$ and $A$.

\textit{Step 3: the contraction mapping.} Let
$\mathcal{F}(w,z,h):=(\tilde w,\tilde z,\tilde h)$; a triple in $\Gamma_T$ is a
fixed point of $\mathcal F$ if and only if it solves \eqref{eq:transform}. From
\eqref{eq:wz}--\eqref{eq:hprime},
\begin{align*}
	\|\tilde{h}' - h^*\|_{C([0,T])} &\le C_2 T^{\alpha/2}, &
	\|\tilde{w} - u_0\|_{C(\Delta_T)} &\le C_1 T^{(1+\alpha)/2}, &
	\|\tilde{z} - v_0\|_{L^\infty(\Delta_T^\infty)} &\le C_1 T^{(1+\alpha)/2},
\end{align*}
so $\mathcal F(\Gamma_T)\subset\Gamma_T$ as soon as
$T \le \min\{C_2^{-2/\alpha}, C_1^{-2/(1+\alpha)}\}$.

Let $(w_i,z_i,h_i)\in\Gamma_T$ and
$(\tilde w_i,\tilde z_i,\tilde h_i)=\mathcal F(w_i,z_i,h_i)$, $i=1,2$. By
\eqref{eq:wz}--\eqref{eq:hprime} these satisfy
\[
\|\tilde{w}_i \|_{C^{(1 + \alpha)/2, 1 + \alpha}(\Delta_T)}\le C_1,
\qquad
\|\tilde{z}_i \|_{C^{(1 + \alpha)/2, 1 + \alpha}(\Delta_T^\infty)}\le C_1,
\qquad
\|\tilde{h}_i' \|_{C^{\alpha/2}([0,T])}\le C_2 .
\]
The function
$W := \tilde{w}_1 - \tilde{w}_2$ satisfies
\[
\begin{aligned}
	& W_t - d_1 B(h_2, y) W_{yy} - \big(d_1C(h_2, y) + h_2' D(h_2, y) \big)W_y \\
	& = d_1[B(h_1, y) - B(h_2, y)] \tilde{w}_{1,yy} + d_1[C(h_1, y) - C(h_2, y) ] \tilde{w}_{1,y} \\
	& \quad + [h_1' D(h_1, y) - h_2' D(h_2, y)]\tilde{w}_{1,y} + f_1 - f_2, \qquad t > 0, \; 0 < y < h_0, \\
	& W_y(t,0) = 0, \quad W(t, h_0) = 0, \quad W(0, y) = 0 .
\end{aligned}
\]
Using the $L^p$ estimates for parabolic equations and Sobolev's embedding
theorem, and the fact that $B,C,D$ are Lipschitz in $h$, we obtain
\begin{equation}\label{eq:w12}
	\|\tilde{w}_1 - \tilde{w}_2 \|_{C^{(1 + \alpha)/2, 1 + \alpha}(\Delta_T)}
	\le C_3 \big(\|w_1 - w_2\|_{C(\Delta_T)} +\|z_1 - z_2\|_{L^\infty(\Delta_T^\infty)}   +\|h_1 - h_2\|_{C^1([0,T])} \big),
\end{equation}
and similarly
\begin{equation}\label{eq:z12}
	\|\tilde{z}_1 - \tilde{z}_2\|_{C^{(1+\alpha)/2, 1+\alpha}(\Delta_T^\infty)}
	\le C_4 \big(\|w_1 - w_2\|_{C(\Delta_T)} + \|z_1 - z_2\|_{L^\infty(\Delta_T^\infty)} + \|h_1 - h_2\|_{C^1([0,T])}\big),
\end{equation}
where $C_3,C_4$ depend on $\alpha, h_0, c, C_1, C_2$, the parameters of the
system, the Lipschitz constants of $A$ and $\mu$ and the cutoff $\zeta$. Since
$\tilde h_i'(t)=-\mu\big(A(h_i(t)-ct)\big)\tilde w_{i,y}(t,h_0)$ and
$\mu\circ A$ is Lipschitz, there is $L_{\mu,A}>0$ with
\begin{equation}\label{eq:hp12}
	\|\tilde{h}'_1 - \tilde{h}'_2\|_{C^{\alpha/2}([0,T])}
	\le L_{\mu,A} \big( \|\tilde{w}_{1,y} - \tilde{w}_{2,y}\|_{C^{\alpha/2,0}(\Delta_T)}
	+ \|h_1 - h_2\|_{C([0,T])}\big).
\end{equation}
Because the differences $\tilde w_1-\tilde w_2$, $\tilde z_1-\tilde z_2$
and $\tilde h'_1-\tilde h'_2$ vanish at $t=0$, the corresponding supremum
norms on $[0,T]$ gain a factor $T^{\alpha/2}$ from their time H\"older
seminorms. Combining this observation with \eqref{eq:h12}, \eqref{eq:w12},
\eqref{eq:z12} and \eqref{eq:hp12}, and assuming $T\le1$, gives
\[
\begin{aligned}
	& \|\tilde{w}_1 - \tilde{w}_2\|_{C^{(1+\alpha)/2, 1+\alpha}(\Delta_T)} + \|\tilde{z}_1 - \tilde{z}_2\|_{C^{(1+\alpha)/2, 1+\alpha}(\Delta_T^\infty)} + \|\tilde{h}'_1 - \tilde{h}'_2\|_{C^{\alpha/2}([0,T])} \\
	& \quad \le C_5 \big(\|w_1 - w_2\|_{C(\Delta_T)} + \|z_1 - z_2\|_{L^\infty(\Delta_T^\infty)} + \|h'_1 - h'_2\|_{C([0,T])}\big),
\end{aligned}
\]
with $C_5$ depending on $C_3, C_4$ and $L_{\mu,A}$. Hence, choosing
\[
T := \min \left\{ 1, \ \left(\tfrac{1}{2C_5}\right)^{2/\alpha}, \ C_2^{-2/\alpha}, \ C_1^{-2/(1+\alpha)}, \ \tfrac{h_0}{8(1+h^*)} \right\},
\]
we get
\[
\begin{aligned}
	& \|\tilde{w}_1 - \tilde{w}_2\|_{C(\Delta_T)} + \|\tilde{z}_1 - \tilde{z}_2\|_{L^\infty(\Delta_T^\infty)} + \|\tilde{h}'_1 - \tilde{h}'_2\|_{C([0,T])}\\
	&\quad \le C_5 T^{\alpha/2} \big(\|w_1 - w_2\|_{C(\Delta_T)} + \|z_1 - z_2\|_{L^\infty(\Delta_T^\infty)} + \|h'_1 - h'_2\|_{C([0,T])}\big) \\
	&\quad \le \tfrac{1}{2} \big(\|w_1 - w_2\|_{C(\Delta_T)} + \|z_1 - z_2\|_{L^\infty(\Delta_T^\infty)} + \|h'_1 - h'_2\|_{C([0,T])}\big).
\end{aligned}
\]
Thus $\mathcal F$ is a contraction on $\Gamma_T$ and has a unique fixed point
$(w,z,h)$. Undoing the change of variables, $(u,v,h)$ solves \eqref{main}, and
parabolic Schauder estimates give $h\in C^{1+\alpha/2}([0,T])$ and,
for every $\tau\in(0,T)$,
\[
 u\in C^{1+\alpha/2,2+\alpha}
 \big(\{(t,x):\tau\le t\le T,\ 0\le x\le h(t)\}\big),
\]
while
$ v\in C^{1+\alpha/2,2+\alpha}([\tau,T]\times[0,R])$ for every
$R>0$.  Thus $(u,v,h)$ is classical for positive time.  No
$C^{2+\alpha}$ compatibility at $t=0$ is asserted. Since all the constants above depend
continuously on the data and on $c$, the solution depends continuously on
$(u_0,v_0,h_0,c,\mu)$, a fact used in \cref{sec_4} and \cref{sec_6}. \qed

\section*{Declarations}

\textbf{Competing interests:} The authors declare that they have no conflict of interest.


\textbf{Data availability.} Data sharing is not applicable to this article, as no datasets were generated or analysed during the current study.

\textbf{Declaration of generative AI and AI-assisted technologies in the manuscript preparation process.} During the preparation of this work, the authors used ChatGPT (OpenAI) to assist with language refinement, manuscript organization, and preliminary review of mathematical arguments. The authors independently reviewed and verified the mathematical content, edited the AI-assisted output as needed, and takes full responsibility for the content of the published article.


\bibliographystyle{abbrvurl}
\bibliography{ref}

\begin{thebibliography}{10}

\bibitem{MR3764580}
W.~Bao, Y.~Du, Z.~Lin, and H.~Zhu.
\newblock Free boundary models for mosquito range movement driven by climate
  warming.
\newblock {\em J. Math. Biol.}, 76(4):841--875, 2018.
\newblock \href {https://doi.org/10.1007/s00285-017-1159-9}
  {\path{doi:10.1007/s00285-017-1159-9}}.

\bibitem{MR2471053}
H.~Berestycki, O.~Diekmann, C.~J. Nagelkerke, and P.~A. Zegeling.
\newblock Can a species keep pace with a shifting climate?
\newblock {\em Bull. Math. Biol.}, 71(2):399--429, 2009.
\newblock \href {https://doi.org/10.1007/s11538-008-9367-5}
  {\path{doi:10.1007/s11538-008-9367-5}}.

\bibitem{DLN2026a}
P.~V. Dinh, P.~Le, and T.~D. Nguyen.
\newblock A free boundary problem for spreading of an invasive species in the
  territory of a native competitor under shifting climate, 2026.
\newblock \href {https://arxiv.org/abs/2609.14573} {\path{arXiv:2609.14573}},
  \href {https://doi.org/10.48550/arXiv.2609.14573}
  {\path{doi:10.48550/arXiv.2609.14573}}.

\bibitem{MR4404216}
Y.~Du.
\newblock Propagation and reaction-diffusion models with free boundaries.
\newblock {\em Bull. Math. Sci.}, 12(1):Paper No. 2230001, 56, 2022.
\newblock \href {https://doi.org/10.1142/S1664360722300018}
  {\path{doi:10.1142/S1664360722300018}}.

\bibitem{MR2607347}
Y.~Du and Z.~Lin.
\newblock Spreading-vanishing dichotomy in the diffusive logistic model with a
  free boundary.
\newblock {\em SIAM J. Math. Anal.}, 42(1):377--405, 2010.
\newblock \href {https://doi.org/10.1137/090771089}
  {\path{doi:10.1137/090771089}}.

\bibitem{MR3327894}
Y.~Du and Z.~Lin.
\newblock The diffusive competition model with a free boundary: invasion of a
  superior or inferior competitor.
\newblock {\em Discrete Contin. Dyn. Syst. Ser. B}, 19(10):3105--3132, 2014.
\newblock \href {https://doi.org/10.3934/dcdsb.2014.19.3105}
  {\path{doi:10.3934/dcdsb.2014.19.3105}}.

\bibitem{DU_MA_2001}
Y.~Du and L.~Ma.
\newblock Logistic type equations on $\mathbb{R}^n$ by a squeezing method
  involving boundary blow-up solutions.
\newblock {\em Journal of the London Mathematical Society}, 64(1):107--124,
  2001.
\newblock \href {https://doi.org/10.1017/S0024610701002289}
  {\path{doi:10.1017/S0024610701002289}}.

\bibitem{MR4873231}
Y.~Du, W.~Ni, and L.~Shi.
\newblock Long-time dynamics of a competition model with nonlocal diffusion and
  free boundaries: vanishing and spreading of the invader.
\newblock {\em SIAM J. Math. Anal.}, 57(2):1195--1226, 2025.
\newblock \href {https://doi.org/10.1137/24M1647473}
  {\path{doi:10.1137/24M1647473}}.

\bibitem{MR3609207}
Y.~Du, M.~Wang, and M.~Zhou.
\newblock Semi-wave and spreading speed for the diffusive competition model
  with a free boundary.
\newblock {\em J. Math. Pures Appl. (9)}, 107(3):253--287, 2017.
\newblock \href {https://doi.org/10.1016/j.matpur.2016.06.005}
  {\path{doi:10.1016/j.matpur.2016.06.005}}.

\bibitem{MR3871607}
Y.~Du, L.~Wei, and L.~Zhou.
\newblock Spreading in a shifting environment modeled by the diffusive logistic
  equation with a free boundary.
\newblock {\em J. Dynam. Differential Equations}, 30(4):1389--1426, 2018.
\newblock \href {https://doi.org/10.1007/s10884-017-9614-2}
  {\path{doi:10.1007/s10884-017-9614-2}}.

\bibitem{MR658490}
P.~Hartman.
\newblock {\em Ordinary differential equations}.
\newblock Birkh\"{a}user, Boston, MA, second edition, 1982.

\bibitem{MR3691993}
H.~Hu and X.~Zou.
\newblock Existence of an extinction wave in the {F}isher equation with a
  shifting habitat.
\newblock {\em Proc. Amer. Math. Soc.}, 145(11):4763--4771, 2017.
\newblock \href {https://doi.org/10.1090/proc/13687}
  {\path{doi:10.1090/proc/13687}}.

\bibitem{HU20205931}
Y.~Hu, X.~Hao, X.~Song, and Y.~Du.
\newblock A free boundary problem for spreading under shifting climate.
\newblock {\em Journal of Differential Equations}, 269(7):5931--5958, 2020.
\newblock URL:
  \url{https://www.sciencedirect.com/science/article/pii/S0022039620301959},
  \href {https://doi.org/10.1016/j.jde.2020.04.024}
  {\path{doi:10.1016/j.jde.2020.04.024}}.

\bibitem{MR241822}
O.~A. Ladyzhenskaya, V.~A. Solonnikov, and N.~N. Ural'tseva.
\newblock {\em Linear and quasilinear equations of parabolic type}, volume~23
  of {\em Translations of Mathematical Monographs}.
\newblock American Mathematical Society, Providence, RI, 1968.
\newblock Translated from the Russian by S. Smith.

\bibitem{MR4766631}
P.~Le and H.-H. Vo.
\newblock A free boundary model for mosquitoes with conditional dispersal in a
  globally unfavorable environment induced by climate warming.
\newblock {\em J. Dynam. Differential Equations}, 36(2):1703--1719, 2024.
\newblock \href {https://doi.org/10.1007/s10884-023-10291-1}
  {\path{doi:10.1007/s10884-023-10291-1}}.

\bibitem{MR3639148}
C.~Lei and Y.~Du.
\newblock Asymptotic profile of the solution to a free boundary problem arising
  in a shifting climate model.
\newblock {\em Discrete Contin. Dyn. Syst. Ser. B}, 22(3):895--911, 2017.
\newblock \href {https://doi.org/10.3934/dcdsb.2017045}
  {\path{doi:10.3934/dcdsb.2017045}}.

\bibitem{Lei2018}
C.~Lei, H.~Nie, W.~Dong, and Y.~Du.
\newblock Spreading of two competing species governed by a free boundary model
  in a shifting environment.
\newblock {\em Journal of Mathematical Analysis and Applications},
  462(2):1254--1282, 02 2018.
\newblock \href {https://doi.org/10.1016/j.jmaa.2018.02.042}
  {\path{doi:10.1016/j.jmaa.2018.02.042}}.

\bibitem{MR1319817}
H.~L. Smith.
\newblock {\em Monotone dynamical systems}, volume~41 of {\em Mathematical
  Surveys and Monographs}.
\newblock American Mathematical Society, Providence, RI, 1995.
\newblock An introduction to the theory of competitive and cooperative systems.

\bibitem{MR3377515}
H.-H. Vo.
\newblock Persistence versus extinction under a climate change in mixed
  environments.
\newblock {\em J. Differential Equations}, 259(10):4947--4988, 2015.
\newblock \href {https://doi.org/10.1016/j.jde.2015.06.014}
  {\path{doi:10.1016/j.jde.2015.06.014}}.

\bibitem{MR3986328}
Z.~Wang, H.~Nie, and Y.~Du.
\newblock Asymptotic spreading speed for the weak competition system with a
  free boundary.
\newblock {\em Discrete Contin. Dyn. Syst.}, 39(9):5223--5262, 2019.
\newblock \href {https://doi.org/10.3934/dcds.2019213}
  {\path{doi:10.3934/dcds.2019213}}.

\end{thebibliography}

\end{document}